\documentclass[11pt]{article}
\usepackage[english]{babel}
\usepackage{amsmath,amsfonts,amssymb,amstext,amsmath,amsthm,amscd}
\usepackage{mathtools}
\usepackage[a4paper,top=2.5cm,bottom=3cm,left=3cm,right=3cm]{geometry}
\usepackage{breakcites}
\usepackage{enumitem}
\usepackage{booktabs}
\usepackage{mathtools}
\usepackage{mathrsfs}
\usepackage{dsfont}

\usepackage{graphicx}
\usepackage{subfigure}
\usepackage{wrapfig}
\usepackage{indentfirst}
\usepackage{latexsym}
\usepackage{sidecap}
\usepackage{booktabs}
\usepackage{hyperref}
\usepackage{lipsum}
\usepackage{cleveref}
\usepackage[dvipsnames]{xcolor}

\usepackage{tikz,tikz-3dplot}
\usepackage{ifthen}
\usetikzlibrary{patterns}% 

\newcommand\myshade{85}
\colorlet{mylinkcolor}{violet}
\colorlet{mycitecolor}{YellowOrange}
\colorlet{myurlcolor}{RoyalBlue}

\hypersetup{
    colorlinks=true,
    linkcolor=myurlcolor,
    citecolor=mycitecolor!\myshade!black,
    filecolor=magenta,      
    urlcolor=magenta,
}
\usepackage{setspace}

\usepackage{mathrsfs}
\usepackage{textcomp}
\usepackage{braket}
\usepackage{colortbl}
\usepackage{bbm}
\usepackage{comment}
\usepackage[colorinlistoftodos]{todonotes}
\newcommand\xqed[1]{%
  \leavevmode\unskip\penalty9999 \hbox{}\nobreak\hfill
  \quad\hbox{#1}}
\newcommand\demo{\xqed{$\triangle$}}

\theoremstyle{plain}
\newtheorem{Theorem}{Theorem}[section]
\theoremstyle{definition}
\newtheorem{Definition}{Definition}[Theorem]
\theoremstyle{remark}
\newtheorem{Remark}[Theorem]{Remark}
\theoremstyle{remark}

\theoremstyle{plain}

\theoremstyle{plain}

\theoremstyle{plain}
\newtheorem{Proposition}[Theorem]{Proposition}
\theoremstyle{remark}
\newtheorem{Example}{Example}[section]
\theoremstyle{remark}

\theoremstyle{remark}

\theoremstyle{remark}

\newcommand\RR{\mathbb{R}}

\newcommand{\Par}[1]{\partial_x  #1}
\newcommand{\XX}{\mathcal{X}}
\newcommand{\dd}{\mathbf{d}}
\newcommand{\ii}{\mathbf{i}}

\DeclareMathAlphabet{\pazocal}{OMS}{zplm}{m}{n}
\newcommand{\La}{\mbox{\pounds}}

\usepackage{tikz}
\usetikzlibrary{arrows}

\newcommand{\PS}[2]{\left \{ #1,#2\right \}}
\newcommand{\TT}{T}
\newcommand*{\Scale}[2][4]{\scalebox{#1}{\ensuremath{#2}}}%

\newcommand{\fc}[1]{\textcolor{black}{#1}}

\usepackage{tikz}
\usetikzlibrary{arrows}

\makeatletter
\pgfdeclareshape{datastore}{
  \inheritsavedanchors[from=rectangle]
  \inheritanchorborder[from=rectangle]
  \inheritanchor[from=rectangle]{center}
  \inheritanchor[from=rectangle]{base}
  \inheritanchor[from=rectangle]{north}
  \inheritanchor[from=rectangle]{north east}
  \inheritanchor[from=rectangle]{east}
  \inheritanchor[from=rectangle]{south east}
  \inheritanchor[from=rectangle]{south}
  \inheritanchor[from=rectangle]{south west}
  \inheritanchor[from=rectangle]{west}
  \inheritanchor[from=rectangle]{north west}
  \backgroundpath{
    \southwest \pgf@xa=\pgf@x \pgf@ya=\pgf@y
    \northeast \pgf@xb=\pgf@x \pgf@yb=\pgf@y
    \pgfpathmoveto{\pgfpoint{\pgf@xa}{\pgf@ya}}
    \pgfpathlineto{\pgfpoint{\pgf@xb}{\pgf@ya}}
    \pgfpathmoveto{\pgfpoint{\pgf@xa}{\pgf@yb}}
    \pgfpathlineto{\pgfpoint{\pgf@xb}{\pgf@yb}}
 }
}
\makeatother

\begin{document}

\title{Stochastic port--Hamiltonian systems}
\author{Francesco Cordoni$^{a}$ \and Luca Di Persio$^{b}$ \and Riccardo Muradore$^{b}$}
\date{}
\maketitle

\renewcommand{\thefootnote}{\fnsymbol{footnote}}
\footnotetext{{\scriptsize $^{a}$ Department of Civil, Environmental and Mechanical engineering, via Mesiano 77, 38123, Trento, Italy}}
\footnotetext{{\scriptsize $^{b}$ Department of Computer Science, University of Verona, Strada le Grazie, 15, Verona, 37134, Italy}}
\footnotetext{{\scriptsize E-mail addresses: francesco.cordoni@unitn.it
(Francesco Cordoni), luca.dipersio@univr.it (Luca Di Persio), (Riccardo Muradore) riccardo.muradore@univr.it}}

\begin{abstract}
In the present work we formally extend the theory of port--Hamiltonian systems to include random perturbations. In particular, suitably choosing the space of flow and effort variables we will show how several elements coming from possibly different physical domains can be interconnected in order to describe a \fc{dynamic system} perturbed by general \fc{continuous} semimartingale. \fc{Relevant enough}, the noise does not enter into the system solely as an external random perturbation, since each port is itself \fc{intrinsically stochastic}. \fc{Coherently to the classical deterministic setting}, we will show how such an approach extends \fc{existing literature of stochastic Hamiltonian systems on} pseudo-Poisson and pre--symplectic \fc{manifolds}. Moreover, we will prove that a power-preserving interconnection of stochastic port--Hamiltonian systems is a stochastic port--Hamiltonian system as well.
\end{abstract}

\textbf{AMS Classification subjects:} 34G20, 34F05, 37N35 \medskip

\textbf{Keywords or phrases: } Stochastic geometric mechanics, port--Hamiltonian systems, stochastic equations on manifold, Dirac manifold.

\maketitle
\tableofcontents

\color{black}     
\section{Introduction}

The mathematical formulation of port--Hamiltonian systems (PHS) and \textit{Dirac manifolds} is long-standing. Starting from its first formulations, \cite{Cou,Dal,Dal2}, it has been generalized along years to cover a heterogeneous set of applications, spanning from passivity-based control of mechanical systems, \cite{OVSME}, to process control,  \cite{RMS}, from mechatronics, \cite{MZ}, to computer science applied to motors related problems, \cite{YYLW}

From a mathematical point of view, the port--Hamiltonian framework is a combination of coordinate--free geometric Hamiltonian dynamics together with a port--modelling perspective. In particular, the equations of motion describing the dynamics of a physical system are given together with the interconnection structure of the network model which provides a geometric structure, known as the associated \textit{Dirac structure}, representing the energetic topology of the system. In particular, a \textit{Dirac structure} can be seen as a generalization of (pseudo) Poisson and pre-symplectic structures. This implies that PHS are primarily geometric objects, whose main and most general representation is implicit and based on a coordinate--free geometric formulation, \cite{Sec,VdSBook,VdSPH}. 

The classical approach to geometric mechanics is given via Poisson and symplectic structures, \cite{Holm1,Holm2,Holm3}. Dirac structures overcome both formulations, \fc{allowing} to describe the underlying structure of the system via a mixed set of differential and algebraic constraints. Therefore, it is  possible to formulate the general notion of \textit{implicit port--Hamiltonian system}, the core of it being represented by the geometric notion of Dirac structure, describing the power interconnection of the system. This is the fundamental reason why the  Dirac structure constitutes the key ingredient for the port--Hamiltonian formalism: it reflects both physical properties and invariants of the system. Moreover, they can also be used to study relevant problems related to {\it non-equilibrium thermodynamics}, \cite{YGB1,YGB2}.

The main goal of the present research is to generalize port--Hamiltonian systems by formally introducing stochastic port--Hamiltonian systems (SPHS). The resulting class of stochastic systems will be shown to be general enough to include stochastic dynamics of (controlled) physical and mechanical systems in a random environment as well as systems characterized by parameters uncertainty that has to be modelled as random variables.

The need for the proposed SPHS generalization is twofold. On one side, even if a deterministic time evolution of a system is assumed, it is often unrealistic to accurate estimate the driving parameters characterizing it, so that we are forced to take measurement errors into account, \cite{Ort,Tsi}. \fc{On the other side}, a system typically interacts with an environment whose behavior and characteristics are not completely known. 
This results in a fundamental ignorance about the real influence of the environment on the system whose dynamics we want to describe. A possible solution to such an issue can be to analyse the external environment \fc{as described by a random vector field}, \cite{Holm4,HolmVP,Ort}. We would also like to mention that high sensitivity of some physical systems to certain \fc{parameters} is often efficiently tackled via probabilistic methods, \cite{Bes,Eyi,Fla}. 

The above reasons demand for a setting which allows to include stochastic Hamiltonian systems and stochastic dynamics on Poisson and symplectic manifolds. Poisson Hamiltonian dynamics has been first \fc{introduced in} the stochastic case in \cite{Bis}, and it has been generalized over the years, see \cite{Holm3,Ort} and the references therein. In particular, such a treatment starts from classical deterministic Hamilton equations of motion that, on a Poisson manifold, read as
\[
\dot{x} = \{x,H\}=: \mathrm{X}_H(x)\,,
\] 
being $\{ \cdot,\cdot\}$ the Poisson bracket, $H$ the \textit{Hamiltonian} of the system, representing the total energy, while $\mathrm{X}_H$ is called \textit{Hamiltonian flow}. Thus, a random perturbation is added to the system considering a \fc{stochastic} Hamiltonian of the form $\hat{H} := H + h \dot{W}$, where $h$ is a suitable function, \fc{typically referred to as stochastic potential}, and $\dot{W}$ is the formal time--derivative of a Brownian motion. \fc{In its most general formulation, as recently introduced in \cite{Ort}}, one can assume that the system is perturbed by a continuous semimartingale, so that the \textit{Hamilton equations of motion} become
\[
\delta X_t= \mathrm{X}_{\hat{H}}(X_t) \delta Z_t\,,
\]
where the notation $\delta X$ indicates that the (stochastic) integration is taken in the Stratonovich sense, see below for further details, while $Z$ is a general semimartingale. 

\fc{Stochastic port--Hamiltonian systems (SPHS)} have been previously studied only in \cite{Had,Sat1,Sat2,Sat3,Sat4}, \fc{and more recently in \cite{CDPM_Bil,CDPMTank,CDPM_IFAC,CDPM_Dis,CDPM_Weak}}. Nonetheless, all of these results start considering an \fc{input--state--output} formulation of the deterministic PHS, then extending the theory from the deterministic to the stochastic setting just adding a random perturbation represented by a standard Brownian motion. In particular, \fc{none of the mentioned} papers address the \fc{founding core of the PHS theory}, namely the Dirac structure. \fc{Therefore, to the best of our knowledge, no implicit formulation for SPHSs has been previously given in literature}.

In what follows, we largely exploit the theory of stochastic differential equations on manifolds, \cite{EmeBook,Hsu}, \fc{and in particular the tools from global stochastic analysis as introduced in \cite{Sch,Mey}}, in connection with the analysis of stochastic Hamiltonian dynamics, \cite{Ort}. In order to generalize the notion of \textit{Dirac structure} and \textit{port-Hamiltonian system}, we will follow an approach similar to the one used in \cite{VdSC2}, to generalize classical deterministic PHS to distributed parameters, \fc{so that} flow and effort variables \fc{are defined by means} of \textit{Stratonovich stochastic vector fields}.

This allows us to generalize existing results on SPHS in several directions. First of all, our stochastic formulation will start at the very core of PHS, i.e., by modelling the Dirac structure and then by defining SPHS as a purely implicit and coordinate--free geometric object. \fc{Therefore, we will be able} to recover the existing notion of SPHS as a {\it particular case}.
Then we shall provide a description allowing the noise to affect the system in different ways. On one side, \fc{each port is by itself intrinsically stochastic and, on the other side,} a stochastic port is added to the whole system, hence describing the noise as an external random vector field affecting the system. The latter description is equivalent to consider the system embedded in an external stochastic environment in which the system itself evolves. It is worth mentioning that such point of view constitutes the typical way in which the noise is considered to enter into systems, see, in particular, the input--state--output SPHS defined in \cite{Sat1,Sat2,Sat3,Sat4,CDPM_Bil,CDPMTank,CDPM_IFAC,CDPM_Dis,CDPM_Weak}, where the the noise is modelled as an external random perturbation. 
Let us \fc{further} note that our formulation also allows for a more general source of randomness. In fact, each element of the system can be considered to be a semimartingale. This means that the noise is not only a possible result of the interaction between the system and an external random \fc{environment:} each port may provide its own random contribution to the whole system. \fc{In this sense, the power exchanged by any port of the SPHS can be a semimartingale itself}. As a byproduct of such an approach, we are also able to treat the noise as an error about parameters.

In order to generalize the well-established theory of deterministic PHS to the stochastic case, we will consider flow variables to be stochastic random fields \fc{perturbed} by a general semimartingale. In what follows we will also use the notation $\mathfrak{X}_{Z^\alpha}(\XX)$, to indicate the space of \textit{(Stratonovich) vector fields} perturbed by the semimartingale $Z^\alpha$ on the manifold $\XX$, so that the flow variable $\delta f^\alpha_t \in \mathfrak{X}_{Z^\alpha}(\XX)$ takes the particular form
\begin{equation}\label{EQN:StraIntro}
\delta f^\alpha_t = e^\alpha(f^\alpha_t,Z^\alpha_t) \delta Z^\alpha_t\,.
\end{equation}
Therefore, our setting generalizes classic deterministic treatment, allowing each port to be a general semimartingale. We remark that, as it is standard in stochastic analysis, equation \eqref{EQN:StraIntro} has to be intended as the short hand notation for
\[
f^\alpha_t - f^\alpha_0 = \int_0^t e^\alpha(f^\alpha_s,Z^\alpha_s) \delta Z^\alpha_s\,,
\]
\fc{being $e^\alpha$ a suitable regular enough function referred to as \textit{Stratonovich operator}, \cite{EmeBook}}. In what follows, even if not specified, we will always consider continuous semimartingale. \fc{Following \cite{EmeBook} it can be seen that the stochastic integral 
\[
P_t := \int_0^t \langle e_s , \delta f_s \rangle\,,
\]
is well-defined and called \textit{Stratonovich integral} of $e$ along the semimartingale $f$. The stochastic integral $P_t$ is a real-valued semimartingale and, as standard in the PHS formalism, it represents the total power exchange through the port. We stress again that, one of the major contribution of the present work is the fact that in complete generality we allow the power exchanged by any port to be a semimartingale. It is worth remarking that, differently from the notation used in the deterministic context, we denote the flow variable by $\delta f_t$ whereas $f_t$ denotes the semimartingale that generates the flow $\delta f_t$. This choice has been done to stress that the flow variable in the proposed setting can be a stochastic vector field integrated in the Stratonovich sense.}

As stated above, the Stratonovich approach to stochastic calculus will be used. In general, when stochastic dynamics is described over general geometric structures, such as manifolds, many problems may arise.
Between them, the choice of the most {\it convenient} or {\it natural} notion of integration to be used. We stress that within stochastic analysis framework, several notions of stochastic integration can be given. This means that, case by case, one usually chooses the most suitable one with respect to the specific mathematical scenario of interest.
As a broad classification, and just to limit ourselves to consider the two most used stochastic theories of integration, it can be said that while \textit{Stratonovich integration} enjoys good geometric properties, the \textit{It\^{o} integral} definition has good probabilistic properties, such as the martingale \fc{property of} the Brownian motion. The geometric nature of Dirac structure suggests \fc{the choice of} Stratonovich \fc{calculus}. The general treatment will be \fc{thus} carried out in such a setting. \fc{To make the treatment as general as possible, we will show how to} {\it translate} Stratonovich stochastic integrals into the corresponding It\^{o} formulation; \fc{we remark that the It\^{o} formulation is extremely useful to obtain certain estimates, for instance to compute conserved physical quantities exploiting general probabilistic properties of the It\^{o} integral}. For such a reason, we will show how SPHS in Stratonovich sense can be converted into the corresponding It\^{o} formulation. We refer  the interested reader to \cite{Oks} for a complete analysis of links and differences between the two different approaches to stochastic integration. Last but not least, let us also underline that some very recent works have appeared attempting to directly use the \textit{It\^{o} integral} formulation from a geometric perspective, see, e.g., \cite{Arm} and the references therein.  

\color{black}
The present work is structured as follows: in Section \ref{SEC:ItoStr} we will introduced main facts and results on stochastic integration on manifolds used throughout the paper; in Section \ref{SEC:PHS} we will recall the main results regarding the theory of deterministic explicit input--state--output port-Hamiltonian systems, starting from the deterministic PHS and then introducing explicit stochastic PHS in Section \ref{SEC:ESPHS}; Section \ref{SEC:IPHS}  will be devoted to generalize previous results to formally define implicit port-Hamiltonian systems seen as power preserving interconnections of certain port elements. Section \ref{SEC:ISPHS} presents the formal definition of stochastic implicit port-Hamiltonian systems and some results are introduced. Subsection \ref{SEC:IntSPHS} studies the interconnected stochastic port-Hamiltonian systems, while Subsection \ref{SSEC:Ito} shows how SPHS, previously considered from the {\it Stratonovich} point of view, can be equivalently defined in terms of It\^{o} integral. Conclusions are drawn in Section \ref{SEC:Conc}.

\color{black}

\section{It\^{o} and Stratonovich calculus on manifolds}\label{SEC:ItoStr}

Before entering into details on the port--Hamiltonian formalism, to make the present work as much self-contained as possible, we will briefly recall the main definition and results on It\^{o} and Stratonovich calculus on manifolds. It is worth stressing that this section does not want to be exhaustive on the topic: we refer the reader to \cite{EmeBook,ElwBook,Hsu,Ort} for a detailed introduction to manifold-valued semimartingales and semimartingale driven Hamiltonian systems. In order to introduce semimartingale-driven SPHS, we will make extensive use of the global stochastic analysis as introduced in \cite{Sch,Mey} and deeply investigated in \cite{Eme}.

Given a general manifold $\XX$, we will denote by $T_x \XX$ the space of tangent vector to $\XX$ at $x \in \XX$ and by $T \XX := \bigcup_{x \in \XX}T_x \XX$ the tangent bundle. The section of the bundle $\XX \to T\XX$ is the space of \textit{(Stratonovich) vector fields} $\mathfrak{X} (\XX)$. Moreover,  $T^*_x \XX$ is the space of cotangent vectors of $\XX$ at $x$ and $T^* \XX := \bigcup_{x \in \XX}T^*_x \XX$ represents the cotangent bundle. The section of the bundle $ \XX \to T^*\XX$ is the space of one-forms $\Omega^1(\XX)$. 

Further, a field of \textit{tangent vectors} of order 2 to a manifold $\XX$ at the point $x$ is a differential operator of order at most 2 with no constant term, that is $L:C^\infty(\XX) \to \RR$ such that
\[
L[f^3](x)=3f(x)L[f^2](x) - 3 f^2(x)L[f](x)\,.
\]

The space of \textit{tangent vectors} of order 2 at $x$ is denoted by $\tau_x \XX$, and the \textit{second order} tangent bundle of $\XX$ is denoted by $\tau \XX := \bigcup_{x \in \XX} \tau_x \XX$. We will denote by $\mathfrak{X}_2(\XX)$ the space of \textit{vector fields} of order 2 which is defined as the section of the tangent bundle $\tau \XX$. Similarly, we can define \textit{forms} of order 2 $\Omega_2(\XX)$ as smooth sections of the cotangent bundle $\tau^* \XX := \bigcup_{x \in \XX} \tau_x^* \XX$. Then, for any function $f \in C^\infty (\XX)$, and $L \in \mathfrak{X}_2(\XX)$, we define the form of order 2 $\dd_2 f \in \Omega_2(\XX)$ as
\[
\dd_2 f (L) := L[f]\,.
\] 
We refer the interested reader to \cite[Chapter 6]{Eme} or also to \cite{Ort} for a detailed introduction to the topic. It can be immediately seen that standard tangent vectors are contained in the tangent vector of order 2, that is $T \XX \subset \tau \XX$, \cite{Eme,EmeBook}.

Exactly as for classical tangent vectors of order 1, forms of order 2 are dual to the space of tangent vectors of order 2. Consequently, we can define a pairing operator $\langle \theta , dX\rangle $ between a $\theta \in \Omega_2(\XX)$ and $dX \in \mathfrak{X}_2(\XX)$. Thus, \cite{Eme}, the map $\theta \mapsto \int_0^t \langle \theta , dX_s\rangle$ is well-defined, and the stochastic integral $\int_0^t \langle \theta_s , dX_s\rangle$ is called \textit{It\^{o} integral} of $\theta$ along $X$. Moreover, by \cite[Theorem 6.24]{EmeBook}, it follows that there exists a unique linear map $\theta \mapsto \int_0^t \langle \theta_s , dX_s\rangle$ associating a continuous real--valued semimartingale to $\theta$.

Thus, for $\alpha \in \Omega^1(\XX)$ and a semimartinale $X$ on the manifold $\XX$, the Stratonovich integral $\int_0^t \langle \alpha , \delta X_s\rangle$ of $\alpha$ along $X$ is defined to be the semimartingale $\int_0^t \langle \dd_2 \alpha , dX_s\rangle$. Concerning the case considered in the present work, it is relevant the $T^*\XX-$valued semimartingales case, to consider stochastic Hamiltonians. In particular, the Stratonovich integral of a $T^*\XX-$valued semimartingale $\beta$ along $X$ is the unique real-valued semimartingale such that the following equalities hold true
\begin{equation}\label{EQN:FormVSem}
\begin{split}
&\int_0^t \langle \dd f,\delta X_s\rangle = f(X_t)-f(X_0)\,,\\
&\int_0^t \langle Z \beta ,\delta X_s\rangle =\int_0^t Z(X_s) \delta \left( \int_0^s \langle \beta ,\delta X_q\rangle \right )\,,
\end{split}
\end{equation}
for any $f \in C^\infty(\XX)$ and any continuous semimartingale $Z$.

Let us introduce the notion of Stratonovich \textit{Stochastic Differential Equations} (SDE) on a manifold, \cite{EmeBook}. Let $\mathcal{M}$ and $\mathcal{N}$ be two manifolds; a \textit{Stratonovich operator} from $\mathcal{M}$ to $\mathcal{N}$ is a family $\left (e(x,z)\right )_{z\in \mathcal{M},x\in \mathcal{N}}$ such that $e(x,z): \TT_z \mathcal{M} \to \TT_x \mathcal{N}$ is a linear and smooth map. The adjoint of $e(x,y)$ is $e^*(x,z):\TT^*_x \mathcal{N} \to \TT^*_z \mathcal{M}$. It is worth noticing that the \textit{Stratonovich operator} $e$ is a map from $T\mathcal{M} \times \mathcal{N}$ to $ T \mathcal{N}$, and $e$ is a section of the fiber bundle $T^* \mathcal{M} \oplus T\mathcal{N}$ over $\mathcal{M}\times \mathcal{N}$.

Given $Z$ a $\mathcal{M}-$valued semimartingale, we will say that the $\mathcal{N}-$valued semimartingale $X$ is the solution  to the \textit{Stratonovich stochastic differential equation}
\begin{equation}\label{EQN:Strat1}
\delta X_t = e(X_t,Z_t)\delta Z_t\,,
\end{equation}
with initial condition $X_0$, if
\begin{equation}\label{EQN:ESPHSIOAdj}
\int_0^t \langle \theta,\delta X_s\rangle = \int_0^t \langle e^*(X_s,Z_s)\theta,\delta Z_s\rangle\,,
\end{equation}
holds $\forall\,\theta \in \Omega^1(\mathcal{N})$, where $\langle \cdot \, \cdot \rangle$ denotes the standard pairing between a form $\theta$ and a vector field $v$, defined as 
\begin{equation}\label{EQN:PairVFF1}
\langle \theta, v\rangle = \ii_{v}\theta\,,
\end{equation}
denoting the insertion of the vector field $v$ into the form $\theta$ according to the standard rule of exterior calculus, \cite{Holm1}, being $\ii$ the \textit{interior product} or \textit{contraction}, \cite[Ch. 3]{Holm2}.

To treat SDE on manifolds in It\^{o} sense, we will make use of the notion of \textit{Schwartz operator} $s$, that is a family $\left (s(x,z)\right )_{x \in \XX, z \in \RR^m})$ such that $s(x,z):\tau_x \XX \to \RR^m$, being $\tau_x \XX$ the vector space of tangent vectors of order 2 to $\XX$ at $x$, see, \cite[Ch. 6]{EmeBook} and \cite[Appendix 6]{Ort}.

Similarly to the case of Stratonovich SDE on a manifold, we will say that, given $Z$ a $\mathcal{M}-$valued semimartingale, the $\mathcal{N}-$valued semimartingale $X$ is the solution  to the \textit{It\^{o} stochastic differential equation}
\begin{equation}\label{EQN:Ito1}
d X_t = s(X_t,Z_t)d Z_t\,,
\end{equation}
with initial condition $X_0$, if
\begin{equation}\label{EQN:ESPHSIOAdj}
\int_0^t \langle \theta,\delta X_s\rangle = \int_0^t \langle s^*(X_s,Z_s)\theta,\delta Z_s\rangle\,,
\end{equation}
holds $\forall\,\theta \in \Omega_2(\mathcal{N})$.

It can be shown, \cite{Eme}, that to any \textit{Stratonovich operator} $e$ can be associated a \textit{Schwartz operator} $s$. Consider $\gamma(t) = (x(t),y(t)) \in \mathcal{M}\times \mathcal{N}$ a smooth curve such that $e(x(t),y(t))(\dot{x}(t)) = \dot{y}(t)$, we can define 
\begin{equation}\label{EQN:Add1}
s(x(t),y(t))(L_{\ddot{x}(t)}) := L_{\ddot{y}(t)}\,,
\end{equation}
where, for any $h \in C^\infty(\mathcal{M})$ and $g \in C^\infty(\mathcal{N})$, we get
\[
\begin{split}
&L_{\ddot{x}(t)} \in \tau_{x(t)}\mathcal{M}\,, \quad L_{\ddot{x}(t)} [h] := \frac{d^2}{dt^2}h(x(t))\,,\\
&L_{\ddot{y}(t)} \in \tau_{y(t)}\mathcal{N}\,, \quad L_{\ddot{y}(t)} [g] := \frac{d^2}{dt^2}g(y(t))\,.
\end{split}
\]

It can be seen that the relation \eqref{EQN:Add1} completely defines $s$ and furthermore that the SDE \eqref{EQN:Strat1} and \eqref{EQN:Ito1} are equivalent.

\color{black}

\section{Explicit input--state--output port-Hamiltonian systems on manifolds}\label{SEC:PHS}

\subsection{Explicit input--state--output deterministic port-Hamiltonian systems}\label{SEC:DESPHS}

\fc{In order to provide a rigorous generalization of PHS able to take into account for stochastic perturbations, we first consider a geometric formulation of PHS. In particular, we exploit the coordinate--free definition of PHS in terms of Poisson or Leibniz brackets.} We would like to underline that this is not the usual starting point in defining PHS; nevertheless, it emphasizes the main features and mathematical aspects that \textit{stochastic PHSs} should enjoy, giving first insights into a general definition of \textit{implicit stochastic PHS}. Therefore, within the present section, we are going to introduce \textit{Hamiltonian dynamics} in Poisson and Leibniz manifolds. \fc{Since latter topic is well established in literature}, we limit ourselves to recall the fundamental results to give the reader a self contained treatment, while we refer to \cite{GBR, Holm1,Holm5,Olv,Vai} for an in-depth analysis of the topic from a pure deterministic perspective. 

Consider a $n-$dimensional differentiable manifold $\XX$ and the space of smooth real functions on $\XX$, $C^\infty \left (\XX\right )$; \fc{we will} denote by 
\[
\PS{\cdot}{\cdot} : C^\infty \left (\XX\right ) \times C^\infty \left (\XX\right ) \to C^\infty \left (\XX\right )\,,
\]
the \textit{Poisson brackets} satisfying bilinearity, skew-symmetry, Jacobi identity and Leibniz rule, \cite{Holm1}.

Properties of the Poisson bracket, and in particular the Leibniz rule, implies that the value $\PS{F}{G}(x)$, with $F$, $G \in C^\infty(\XX)$, \fc{$x \in \XX$}, depends on both arguments only through the derivative. We can thus associate to a Poisson bracket a controvariant skew--symmetric 2-tensor called \textit{Poisson tensor}
\[
B(x):\Omega^1\left (\XX\right ) \times \Omega^1\left (\XX\right ) \to C^\infty\left (\XX\right )\,,
\] 
defined as
\[
B(x)(\dd F,\dd G)=\PS{F}{G}(x)\,,\quad F,\, G \in C^\infty\left (\XX\right )\,,
\]
where $\dd F:= \partial_{x^i} F dx^i$ and $\dd G:= \partial_{x^i}G d x^i$ are the exterior derivatives of the functions $F$ and $G \in C^\infty \left (\XX\right )$, \fc{respectively}, \cite[Ch. 3]{Holm1}, \fc{having shorthand denoted} by $\partial_{x^i}$ the partial derivative w.r.t. $x^i$, i.e. $\partial_{x^i} := \frac{\partial}{\partial x^i}$, \fc{while} by $\partial_x = (\partial_{x^1},\dots,\partial_{x^n})$ the gradient.

To a \textit{Poisson tensor} we can \fc{associate} a morphisms
\[
B^{\#}(x) : \TT^* \XX \to \TT \XX\,,
\]
defined as
\begin{equation}\label{EQN:BAsh}
B(x)(\dd F,\dd G)=\langle\dd F(x), B^{\#}(\dd G(x))\rangle\,.
\end{equation}

A \textit{Hamiltonian system} on a Poisson manifold $\left (\XX,\PS{\cdot}{\cdot}\right )$ with \textit{Hamiltonian function} $H \in C^\infty\left (\XX\right )$ is thus defined by the differential equation
\begin{equation}\label{EQN:HE}
\dot{x} = \PS{x}{H} = B^{\#}(\dd H) =: \mathrm{X}_H(x)\,.
\end{equation}

Equation \eqref{EQN:HE} is called \textit{Hamilton equations} of motion and $\mathrm{X}_H$ is called \textit{Hamiltonian vector field} generated by \fc{the Hamiltonian} $H$. In particular, \cite[Ch. 4]{Holm1}, equation \eqref{EQN:HE} is equivalent to requiring
\begin{equation}\label{EQN:HE2}
\dot{F} = \PS{F}{H}\,,
\end{equation}
for all differentiable functions $F:\TT^* \XX \to \RR$. 

\textit{Hamilton equations} of motion \eqref{EQN:HE} can be further generalized to define an \textit{(explicit) input--state--output port-Hamiltonian system} (PHS) on a Poisson manifold $\left (\XX,\PS{\cdot}{\cdot}\right )$ with \textit{Hamiltonian function} $H \in C^\infty\left (\XX\right )$ as
\begin{equation}\label{EQN:EPHS}
\begin{cases}
\dot{x} = \mathrm{X}_H(x) + \sum_{i=1}^m u_i \mathrm{X}_{H_{g_i}}(x)\,,\\
y_i = \PS{H}{H_{g_i}}\,,
\end{cases}
\end{equation}
with $x \in \RR^n$ and where $\mathrm{X}_{H_{g_i}}$ is the \textit{Hamiltonian vector field} associated to the Hamiltonian $H_{g_i}$, $u_i \in U$ denotes the $i-th$ input and $y_i \in U^*$ is the $i-th$ output of the system, \cite{Leu,Tab}.

\fc{Using the properties of the Poisson bracket} the \textit{(explicit) input--state--output port-Hamiltonian system} PHS \eqref{EQN:EPHS} can be expressed in local coordinates as
\begin{equation}\label{EQN:EPHSc}
\begin{cases}
\dot{x} = J(x) \partial_x H + \sum_{i=1}^m u_i g_i(x)\,,\\
y_i = g_i^T(x) \partial_x H\,,
\end{cases}
\end{equation}
where $J$ is a skew-symmetric \fc{structure} matrix of suitable dimensions and $g_i$ are $m$ suitable regular enough functions, \cite{Leu}.

We can further include dissipation into the PHS \eqref{EQN:EPHS} by considering
\begin{equation}\label{EQN:ExplLeib1}
\begin{cases}
\dot{x} = \mathrm{X}_H(x) + \sum_{i=1}^m u_i \mathrm{X}_{H_{g_i}}(x) + u^R \mathrm{X}_{H_{g^R}}\,,\\
y_i = \PS{H}{H_{g_i}}\,,\\
y^R = \PS{H}{H_{g^R}}\,,
\end{cases}
\end{equation}
\fc{where $u^R$ describes the} dissipation relation $u^R = \tilde{R}(x)y^R$, with $\tilde{R}$ symmetric and positive semi--definite.

Defining the \textit{Leibniz bracket} for $F$, $G \in C^\infty \left (\XX\right )$ as
\begin{equation}\label{EQN:LeibB}
[F,G]_{L} = B(F,G) - \langle \dd F,\tilde{R}(x)\langle \dd G,g\rangle g\rangle\,,
\end{equation}
and setting the structure matrix as
\begin{equation}\label{EQN:StructLeib}
J(x)-(g^R(x))^T \tilde{R}(x)g^R(x)\,,
\end{equation}
we can define the \textit{(explicit) input--state--output port-Hamiltonian system with dissipation} to be
\begin{equation}\label{EQN:EPHSD}
\begin{cases}
\dot{x} = \mathrm{X}^L_H(x)  + \sum_{i=1}^m u_i \mathrm{X}_{H_{g_i}}(x)\,,\\
y_i = \{H,H_{g_i}\}\,,
\end{cases}
\end{equation}
\fc{where $\mathrm{X}^R_H$ is now the \textit{Hamiltonian vector field} with dissipation defined by the \textit{Leibniz bracket}}
\[
\fc{\mathrm{X}^L_H(\cdot) := [\cdot,H]_{L}}\,.
\]

\color{black}
From equation \eqref{EQN:LeibB} it can be seen that the \textit{Leibniz bracket} is composed by a skew-symmetric part $B$ and a symmetric positive semi--definite part $\langle \dd F,\tilde{R}(x)\langle \dd G,g\rangle g\rangle$. It thus follows, using equation \eqref{EQN:ExplLeib1}, that
\begin{equation}\label{EQN:PassLeib}
\dot{H}(x(t)) = [H,H]_{L}(x(t)) + \sum_{i=1}^m u_i \{H,H_{g_i}\}(x(t)) \leq y^T(t) u(t)\,.
\end{equation}

Equation \eqref{EQN:PassLeib} is known as the \textit{passivity property} and broadly states that the energy variation of the system $\dot{H}(x(t))$ cannot be greater than the energy supplied to the system $ y^T(t) u(t)$. Such property is crucial in several engineering systems and it is extensively used to control purposes, see, e.g.,  \cite{VdSBook,Sec}. In the case of a purely skew--symmetric bracket, i.e. $\tilde{R} = 0$, the energy is conserved and we recover the controlled Poisson dynamics \eqref{EQN:EPHS} so that the inequality in equation \eqref{EQN:PassLeib} becomes an equality. In such a case the system is said to be \textit{lossless}. These properties are at the very core of port--Hamiltonian formulation and they cannot be straightforwardly generalizable to the stochastic case, indeed particular care must be taken when noise enters into system.

\color{black}
From the structure matrix for the \textit{Leibniz bracket} \eqref{EQN:StructLeib} we have that in local coordinates the PHS \eqref{EQN:EPHSD} becomes
\begin{equation}\label{EQN:EPHSDc}
\begin{cases}
\dot{x} = (J(x)-R(x)) \partial_x H(x) + \sum_{i=1}^m u_i g_i(x)\,,\\
y_i = g_i^T(x) \partial_x H(x)\,,
\end{cases}
\end{equation}
with $R(x):=(g^R(x))^T \tilde{R}(x)g^R(x)$.

\subsection{Explicit input--state--output stochastic port-Hamiltonian systems}\label{SEC:ESPHS}

\fc{We are now in position to generalize the notion of (explicit) input--state--output port-Hamiltonian systems to the (explicit) \textit{input--state--output stochastic port-Hamiltonian systems}}. We will consider a filtered \fc{and} complete probability space $\left (\Omega,\mathcal{F},\left (\mathcal{F}_t\right )_{t \in \RR_+},\mathbb{P}\right )$ satisfying standard assumptions, namely right--continuity and saturation by $\mathbb{P}$--null sets.

\fc{As stated in Section \ref{SEC:ItoStr}}, we will denote by $\delta Z$ the integration in the sense of Stratonovich \fc{along the semimartinagle $Z$}, and by $d Z$ the integration in the sense of It\^{o}. The primary motivation \fc{in} using Stratonovich stochastic calculus  \fc{is by its enjoyed good geometric properties, particularly concerning the fact that standard chain rule of calculus holds}. This will allows us to prove one of the main properties \fc{characterizing} port--Hamiltonian systems, \fc{namely:}  energy conservation. Further, it can be shown that any stochastic integral in Stratonovich form can be converted into a corresponding \fc{ It\^{o} integral. Therefore}, in what follows we will show how an analogous treatment can be done using integration in the sense of It\^{o}.

\color{black}
\fc{Manifold-valued SDE, as discussed briefly in Section \ref{SEC:ItoStr},} allows us to introduce the following generalization to consider semimartingale perturbed input--state--output PHS on a Poisson--manifold. Consider $(\XX,\PS{\cdot}{\cdot})$ to be a Poisson manifold, an \textit{(explicit) input--output stochastic port-Hamiltonian system} with Hamiltonian function $H: \XX \to \RR$, stochastic potential $H_N: \XX \to \RR$ and \fc{driving stochastic martingales} $Z$, $Z^{N}$ and $Z^g$, is defined as the solution to the manifold-valued SDE 
\begin{equation}\label{EQN:ESPHSIO}
\begin{cases}
\delta X_t = \mathrm{X}_{H}(X_t) \delta Z_t + u \mathrm{X}_{H_{g}}(X_t) \delta Z_t^g + \mathrm{X}_{H_{N}}(X_t) \delta Z^{N}_t \,,\\
y_t = \PS{H}{H_g}\,,
\end{cases}
\end{equation}
where the vector fields have been defined in terms of the Poisson bracket as with the deterministic input--state--output PHS \eqref{EQN:EPHS}
\[
\begin{split}
\mathrm{X}_{H} (\cdot) &:= \PS{\cdot}{H} \,,\\
\mathrm{X}_{H_g} (\cdot) &:= \PS{\cdot}{H_g} \,,\\
\mathrm{X}_{H_N} (\cdot) &:= \PS{\cdot}{H_N} \,.
\end{split}
\]

Using the fact that $T_z \RR^3 \simeq \RR^3$, we can define the Stratonovich operator
\[
\begin{split}
&e(x,z):\RR^{3} \to \TT_x \XX\,,\\
&e(x,z)(r_0,r_{N},r_g):= r_0 \mathrm{X}_{H}(x) + r_{N} \mathrm{X}_{H_{N}}(x) + r_{g} u \mathrm{X}_{H_g}(x)\,,
\end{split}
\]
so that equation \eqref{EQN:ESPHSIO} can be compactly rewritten c as
\[
\begin{cases}
\delta X_t = e(Z_t,X_t) \delta \mathbf{Z}_t\,,\\
y_t = \PS{H}{H_g}\,,
\end{cases}
\]
with $\mathbf{Z}_t := (Z_t,Z^{N}_t,Z^g_t)$. The adjoint of the Stratonovich operator $e$ is given by
\[
\begin{split}
&e^*(x,z): \TT_x^* \XX \to \RR^{3}\,,\\
&e^*(x,z)(\theta) := -\dd \hat{H} (B^{\#}(\theta))(x) \,,
\end{split}
\]
where we have defined for short
\[
\hat{H}:= H + H_{N} + u H_g\,.
\]

According to equation \eqref{EQN:ESPHSIOAdj}, the  semimartingale solution $X$ to equation \eqref{EQN:ESPHSIO} must be intended as
\begin{equation}\label{EQN:ESPHSIOAdj2}
\int_0^t \langle \theta,\delta X_s\rangle = -\int_0^t \langle \dd \hat{H} (B^{\#}(\theta))(X_s)\theta,\delta \mathbf{Z}_s\rangle\,,
\end{equation}
for any $\theta \in \Omega^1(\XX)$.

\color{black}

\begin{Remark}
The generalization of equation \eqref{EQN:ESPHSIO} to the multi--input multi--output case \fc{yields}
\begin{equation}\label{EQN:ESPHSIOMIO}
\begin{cases}
\delta X_t = \mathrm{X}_{H}(X_t) \delta Z_t + \sum_{i=1}^{m} u_i \mathrm{X}_{H_{g_i}}(X_t) \delta Z_t^{g_i} + \sum_{j=1}^l \mathrm{X}_{H_{N}^j}(X_t) \delta Z^{N_j}_t\,,\\
y^i_t = \PS{H}{H_{g_i}}\,.
\end{cases}
\end{equation}

\color{black}
As a very particular case, consider the case of an autonomous system, i.e., $u \equiv 0$, so that equation \eqref{EQN:ESPHSIOMIO} reads as
\begin{equation}\label{EQN:ESPHSIOMI}
\delta X_t = \mathrm{X}_{H}(X_t) \delta Z_t+ \sum_{j=1}^l \mathrm{X}_{H_{N}^j}(X_t) \delta Z^{N_j}_t\,.
\end{equation}

As before, we can introduce the Stratonovich operator
\[
e(x,z)(r_0,r_{N}):= r_0 \mathrm{X}_{H}(x) + \sum_{j=1}^m r^j_{N} \mathrm{X}_{H_{N}^j}(x) \,,
\]
and write equation \eqref{EQN:ESPHSIOMI} for short as 
\[
\delta X_t = e(Z_t,X_t) \delta \mathbf{Z}_t\,,
\]
with $\mathbf{Z}_t = (Z_t,Z^{N}_t)$. Equation \eqref{EQN:ESPHSIOMI} coincides exactly with the stochastic Hamilton equations of motion on a Poisson manifold as defined in \cite{Ort}.
\color{black}

\demo  
\end{Remark}

To take into account dissipation in the explicit SPHS, following \cite{OrtLei} \fc{, we can introduce a tensor} map $B_{L}:\TT^* \XX \times \TT^* \XX \to \RR$ defined as
\begin{equation}\label{EQN:LeiMor22}
B_{L}(\dd F,\dd G):=[F,G]_{L}\,;
\end{equation}
to which we can associate a vector bundle $B_{L}^{\#}:\TT^* \XX \to \TT \XX$ by the relation
\begin{equation}\label{EQN:LeiMor23}
B_{L}(\dd F,\dd G)=\langle\dd F,B_{L}^{\#}(\dd G)\rangle\,.
\end{equation}

\color{black}
Consider a Leibniz manifold $(\XX,[\cdot,\cdot]_{L})$, an \textit{(explicit) stochastic input--state--output port-Hamiltonian system} with dissipation and with Hamiltonian function $H: \XX \to \RR$, stochastic potential $H_N: \XX \to \RR$ and \fc{driving stochastic martingales} $Z$, $Z^{N}$ and $Z^g$, is defined as the solution to the manifold-valued SDE 
\begin{equation}\label{EQN:ESPHSDiss}
\begin{cases}
\delta X_t = X_{H}^L(X_t) \delta Z_t + u X_{H_{g}}^L(X_t) \delta Z_t^g + X_{H_{N}}^L(X_t) \delta Z^{N}_t \,,\\
y_t = [H,H_g]_L\,,
\end{cases}
\end{equation}
where
\[
\begin{split}
\mathrm{X}_{H}^L (\cdot) := [\cdot,H]_L \,,\quad \mathrm{X}_{H_g}^L (\cdot) := [\cdot,H_g]_L \,,\quad \mathrm{X}_{H_N}^L (\cdot) := [\cdot,H_N]_L \,.
\end{split}
\]

Analogously as seen for the Poisson case, equation \eqref{EQN:ESPHSDiss} can be written in terms of a Stratonovich operator $e$ whose adjoint is given by
\[
\begin{split}
&e^*(x,z): \TT_x^* \XX \to \RR^{3}\,,\\
&e^*(x,z)(\theta) := -\dd \hat{H} (B_{L}^{\#}(\theta))(x) \,,
\end{split}
\]
where
\[
\hat{H}:= H + H_{N} + u H_g\,.
\]

According to \eqref{EQN:ESPHSIOAdj}, the solution semimartingale $X$ to \eqref{EQN:ESPHSIO} has to satisfy
\begin{equation}\label{EQN:ESPHSIOAdj2Diss}
\int_0^t \langle \theta,\delta X_s\rangle = -\int_0^t \langle \dd \hat{H} (B_L^{\#}(\theta))(X_s)\theta,\delta \mathbf{Z}_s\rangle\,,
\end{equation}
for any $\theta \in \Omega^1(\XX)$.

\color{black}
As in the deterministic case, from the structure matrix for the \textit{Leibniz bracket} \eqref{EQN:StructLeib} we have that in local coordinates the SPHS \eqref{EQN:ESPHSDiss} becomes
\begin{equation}\label{EQN:ESPHSDc}
\begin{cases}
\delta X_t = (J(X_t)-R(X_t)) \partial_x H(X_t)\delta Z_t + u g(X_t)\delta Z^g_t + \xi(X_t) \delta Z_t^N\,,\\
y_t = g^T(X_t) \partial_x H(X_t)\,,
\end{cases}
\end{equation}
with $R(x):=(g^R(x))^T \tilde{R}(x)g^R(x)$.

\begin{Example}
In this example we compare the deterministic modeling of a $n-$degree of freedom ($n-$DOF) manipulator with its stochastic version where noise is included into the system.

\begin{description}[style=unboxed,leftmargin=0cm]
\item[(i) - Deterministic $n-$DOF.] Let consider the $n$-degree of freedom ($n$-DOF) and gravity-compensated manipulator
\begin{equation}\label{EQN:NDOFEx}
M(q)\ddot{q} + C(q,\dot{q}) \dot{q} + R \dot{q} = u \,,\\
\end{equation}
where $q = (q^1,\dots,q^n) \in \RR^n$ is the set of generalized coordinates, $C$ is the Coriolis and centrifugal term and $u$ is the generalized command force, \cite{de2012theory}. 

By introducing the generalized momentum $p = M(q)\dot{q}$, $p = (p^1,\dots,p^n) \in \RR^n$ the system \eqref{EQN:NDOFEx} can be rewritten using the SPHS formalism \cite{Sec} as
\begin{equation}\label{EQN:SPHSET}
\begin{cases}
\begin{pmatrix}
\dot{q}(t)\\
\dot{p}(t)
\end{pmatrix} &=
\left (
\begin{pmatrix}
0 & I\\
-I & -R (q,p)
\end{pmatrix}
\begin{pmatrix}
\partial_{q} H \\
\partial_{p} H
\end{pmatrix}
+
\begin{pmatrix}
0\\
g
\end{pmatrix}
u \right ) \,,\\
y&=
\begin{pmatrix}
0 & g^T
\end{pmatrix}
\begin{pmatrix}
\partial_{q} H \\
\partial_{p} H
\end{pmatrix}
 = M^{-1}(q) p = \dot{q}\,,
\end{cases}
\end{equation}
where $R(q,p)$ is the dissipation matrix and
\[
H(q,p) = \frac{1}{2}p^T M^{-1}(q) p\,,
\] 
is the kinetic energy. Since the robot is gravity-compensated we do not have to consider the potential energy in the definition of the Hamiltonian $H$. It is trivial to see that equation \eqref{EQN:SPHSET} can be written in the form of equation \eqref{EQN:EPHSDc} with $x := (q,p)$ as
\[
\begin{cases}
\dot{x} = (J(x)-R(x)) \partial_x H(x) + g(x) u\,,\\
y = g^T(x) \partial_x H(x)\,,
\end{cases}
\]

\item[(ii) - Stochastic $n-$DOF.] Usually, the $n-$DOF system \eqref{EQN:NDOFEx} interacts with an unknown external environment, so that a new term is added in the r.h.s. of equation \eqref{EQN:NDOFEx}. The effect of the environment on the system is often unknown so it can be modelled as a stochastic process. The most classical assumption is that the environment is described by a Brownian motion $W$, \cite{CDPMTank}, and equation \eqref{EQN:SPHSET} can be compactly written as
\begin{equation}\label{EQN:ESPHSDc22}
\begin{cases}
\delta X_t = (J(X_t)-R(X_t)) \partial_x H(X_t)\delta t + g(X_t)u\delta t + \xi(X_t) \delta W_t\,.\\
y_t = g^T(X_t) \partial_x H(X_t)\,,
\end{cases}
\end{equation}
where $X$ is the stochastic counterpart of $x$.

Such equation recovers, apart from the choice of integration, the classical definition of SPHS in \cite{Sat1,Sat2,Sat3,Sat4, CDPM_IFAC,CDPMTank}.

Nonetheless, more general types of noise can be considered so that a general semimartingale $Z^N$ may replace the Brownian motion, obtaining 
\begin{equation}\label{EQN:ESPHSDc222}
\begin{cases}
\delta X_t = (J(X_t)-R(X_t)) \partial_x H(X_t)\delta t + g(X_t)u\delta t + \xi(X_t) \delta Z^N_t\,.\\
y_t = g^T(X_t) \partial_x H(X_t)\,.
\end{cases}
\end{equation}

Equation \eqref{EQN:ESPHSDc222} coincides with \eqref{EQN:ESPHSDc} where $Z$ and $Z^g$ are the deterministic processes given by $(t,\omega)\mapsto t$.

\item[(iii) - Stochastic $n-$DOF with stochastic Hamiltonian.] We can assume that also the energy of the system, and thus the Hamiltonian, is perturbed by a stochastic semimartingale $Z$, so that we obtain the more general case of 
\begin{equation}\label{EQN:ESPHSDc2223}
\begin{cases}
\delta X_t = (J(X_t)-R(X_t)) \partial_x H(X_t)\delta Z_t + g(X_t) u\delta t + \xi(X_t) \delta Z^N_t\,.\\
y_t = g^T(X_t) \partial_x H(X_t)\,.
\end{cases}
\end{equation}

\item[(iv) - Stochastic $n-$DOF with stochastic Hamiltonian and stochastic control.] We can finally assume that also the control is perturbed by a stochastic noise recovering the most general formulation as in \eqref{EQN:ESPHSDc},
\begin{equation}\label{EQN:ESPHSDc22234}
\begin{cases}
\delta X_t = (J(X_t)-R(X_t)) \partial_x H(X_t)\delta Z_t + u g(X_t)\delta Z^g_t + \xi(X_t) \delta Z^N_t\,.\\
y_t = g^T(X_t) \partial_x H(X_t)\,.
\end{cases}
\end{equation}
\end{description}
\demo 
\end{Example}

Before entering into details concerning the energy conservation property of SPHS, we prove the following change of variable formula.

\begin{Proposition}\label{PRO:ChangeVar}
Let $X$ be the solution to the SPHS \eqref{EQN:ESPHSDiss}, then for any $\varphi \in C^{\infty}(\XX)$ it holds
\begin{equation}\label{EQN:ChangeVar}
\begin{cases}
\delta \varphi(X_t) = [\varphi,H]_L (X_t)\delta Z_t + u [\varphi,H_g]_L(X_t) \delta Z^g_t + [\varphi,H_{N}]_L (X_t)\delta Z^{N}_t \,,\\
y_t =  [H,H_g]_L\,.
\end{cases}
\end{equation}
\end{Proposition}
\begin{proof}
Notice that, using \cite[Prop. 7.4]{EmeBook}, it holds
\[
\int_0^t \langle \dd f, \delta X_s\rangle = f(X_t)-f(X_0)\,.
\]
Taking thus $\theta = \dd \varphi$ in equation \eqref{EQN:ESPHSIOAdj2Diss}, we have that
\[
\begin{split}
-& \int_0^t \langle \dd \hat{H}\left (B_L^{\#}(\dd \varphi)\right )(X_s),\delta Z_s\rangle =\\
&= -\int_0^t \langle \langle \dd H, B_L^{\#}(\dd \varphi)\rangle (X_s),\delta Z_s\rangle +\\
&-\int_0^t \langle \langle \dd H_{N}B_L^{\#}(\dd \varphi)\rangle(X_s),\delta Z^{N}_s\rangle + \\
&- \int_0^t u \langle \langle \dd H_g B_L^{\#}(\dd \varphi)\rangle(X_s),\delta Z^C_s\rangle =\\
&= [\varphi,H]_L (X_t)\delta Z_t + [\varphi,H_{N}]_L (X_t)\delta Z^{N}_t +  u [\varphi,H_g]_L(X_t) \delta Z^C_t\,,
\end{split}
\]
where the last equality follows from equations \eqref{EQN:LeiMor22}--\eqref{EQN:LeiMor23}.
\end{proof}

Concerning energy conservation and passivity discussed in equation \eqref{EQN:PassLeib} in the deterministic case, their generalizations to the stochastic case are not trivial. Broadly speaking, the noise can inject energy into the system so that specific conditions on the noise must be imposed to obtain losslessness and passivity. In particular, due to the presence of the semimartingale $Z$, SPHS \eqref{EQN:ESPHSDiss} is not dissipative under standard requirement of the structure matrix $R$ being symmetric and positive semi-definite. This aspects will play a central role in developing some aspects of implicit SPHS and it will be clarified in subsequent sections. It is worth stressing that it is difficult to obtain specific conditions under which passivity or energy conservation holds for the general case; we will limit ourselves to underline the main aspects and more detailed will be given later on when implicit stochastic PHS will be introduced.

Three relevant considerations are in order regarding energy conservation and passivity:
\begin{description}[style=unboxed,leftmargin=0cm]
\item[(i)] Consider the case of a stochastic system with no dissipation and no external control, so that we recover the case of an autonomous stochastic Hamiltonian system on a Poisson manifold. From a physical point of view, it is natural to look for conditions under which $\mathbb{P}-$a.s. energy conservation holds, that is,
\[
\delta H(X_t) = 0\,.
\]
In the deterministic case, it is trivial to see that energy conservation holds due to the skew-symmetric property of the Poisson bracket. As deeply argued in \cite{Ort}, the presence of a stochastic Hamiltonian can destroy the energy conservation property of the system, because
\[
\delta H(X_t) = \{H,H_N\}\delta Z_t^N\,.
\]
To obtain the energy conservation property the additional condition that the stochastic potential $H_N$ must be an involution w.r.t. the Hamiltonian $H$, meaning that $\{H,H_N\}=0$, is required.

\item[(ii)] Concerning energy conservation, it is often more realistic to study when a weaker notion of energy conservation holds, named \textit{weak energy conservation} and defined as
\begin{equation}\label{EQN:WeakDef}
\mathbb{E}X_t - \mathbb{E}X_0 = 0\,.\\
\end{equation}

Weak energy conservation is easier to be satisfied in real application and for this reason it is the most natural definition of energy conservation usually considered for stochastic systems.

\item[(iii)] Consider now the case when no external noise is considered in equation \eqref{EQN:ChangeVar}, i.e., $H_N \equiv 0$, and the semimartingale perturbing the control is the trivial deterministic semimartingale $Z^g_t := t$. Recalling that the \textit{Leibinz bracket} is decomposed in the present case into a purely skew--symmetric bracket $\{\cdot,\cdot\}_{\mbox{skew}}$ and a symmetric positive semi--definite bracket $\{\cdot,\cdot\}_{\mbox{\mbox{sym}}}$, that is,
\[
[\cdot,\cdot]_{L} = \{\cdot,\cdot\}_{\mbox{skew}} - \{\cdot,\cdot\}_{\mbox{sym}}\,.
\]
Then by, equation \eqref{EQN:ChangeVar} with $\varphi = H$, we have
\[
\begin{cases}
\delta H(X_t) &= -\{H,H\}_{\mbox{sym}} (X_t)\delta Z_t + u [\varphi,H_g]_L(X_t) \delta t\,,\\
y_t &=[\varphi,H_g]_L\,.
\end{cases}
\]

From a control perspective, it is desirable to require the system to be $\mathbb{P}-$a.s. passive, that is,
\[
\delta H(X_t) \leq y^T_t u_t \delta t\,.
\]
Analogously to strong energy conservation, we will refer to the above condition as \textit{strong passivity}. Notice that, even if $\{\cdot,\cdot\}_{\mbox{sym}}$ is symmetric positive semi--definite, the presence of the semimartingale $Z$ does not allow to infer that
\begin{equation}\label{EQN:PosSemi}
\{H,H\}_{\mbox{sym}} (X_t)\delta Z_t(\omega) \geq 0\,.
\end{equation}

Condition \eqref{EQN:PosSemi} does not hold for the vast majority of relevant examples, where even the most trivial case of a Brownian motion $B$ does not satisfy such inequality $\mathbb{P}-$a.s. Consequently, usually within the stochastic setting, a weaker notion, named \textit{weak passivity}, is considered
\begin{equation}\label{EQN:WeakDefP1}
\mathbb{E}X_t - \mathbb{E}X_0 \leq \mathbb{E} y^T_t u_t\,.
\end{equation}

It is immediate to see that \textit{weak passivity} \eqref{EQN:WeakDefP1} is valid in a much broad range of situations and for this reason the weak notion is usually considered in literature. This topic will be expanded and treated in more details later as it plays a key role in the implicit definition of a stochastic PHS.
\end{description}

\color{black}
\subsection{It\^{o} explicit input--state--output stochastic port-Hamiltonian systems}

In the present Section we will show how the SPHS in the Stratonovich form can be converted into the corresponding SPHS in It\^{o} form.

\begin{Proposition}\label{PRO:SPHSTOIto}
Let $X$ be the solution to the Stratonovich PHS \eqref{EQN:ChangeVar}, and let $Z$, $Z^{N}$ and $Z^C$ such that
\begin{equation}\label{EQN:NullQC1}
\langle Z,Z^C\rangle_t = \langle Z,Z^{N}\rangle_t = \langle Z^{N},Z^C\rangle_t = 0\,,
\end{equation}
being $\langle Z^j,Z^i\rangle_t$ the quadratic covariation between $Z^i$ and $Z^j$ at time $t$. 

Then $X$ admits an equivalent It\^{o} formulation as
\begin{equation}\label{EQN:ItoSPHSKer1}
\begin{cases}
d \varphi(X_t) &= [\varphi,H]_L (X_t)d Z_t + [[\varphi,H]_L,H]_L d \langle Z,Z\rangle_t+\\
&+u [\varphi,H_g]_L(X_t) d Z^g_t + u[[\varphi,H_g]_L,H_g]_L d \langle Z^g,Z^g\rangle_t+\\
&+[\varphi,H_{N}]_L (X_t)d Z^{N}_t + [[\varphi,H_N]_L,H_N]_L d \langle Z^N,Z^N\rangle_t\,,\\
y_t &= [H,H_g]_L\,.
\end{cases}
\end{equation}
\end{Proposition}
\begin{Remark}
It is worth stressing that conditions \eqref{EQN:NullQC1} are not necessary to prove the It\^{o} representation \eqref{EQN:ItoSPHSKer1}. In fact, a similar result holds dropping such conditions and adding cross terms to equation \eqref{EQN:ItoSPHSKer1}. The choice of assuming conditions \eqref{EQN:NullQC1} has been purely made to avoid heavy notation.

\demo
\end{Remark}
\begin{proof}
Consider a second order vector $L_{\ddot{v}}\in \RR^{m}$, so that we have
\begin{equation}\label{EQN:ConvItoSPHS1}
\begin{split}
s(x,z)(L_{\ddot{z}})[\varphi] &= \left . \frac{d}{dt} \right |_{t=0} \langle \dd \varphi(x(t)), \dot{x}(t) \rangle =  \left . \frac{d}{dt} \right |_{t=0} \langle \dd \varphi(x(t)), e(x(t),z(t)) \dot{z}(t) \rangle = \\
&=\left . \frac{d}{dt} \right |_{t=0} \dot{z}\langle \dd \varphi(x(t)), \mathrm{X}^L_H(x(t))\rangle +\\
&+\left . \frac{d}{dt} \right |_{t=0} \dot{z}^{g} \langle \dd \varphi(x(t)), \mathrm{X}^L_{H_g}(x(t))u_t  \rangle +\\
&+ \left . \frac{d}{dt} \right |_{t=0} \dot{z}^{N} \langle  \dd \varphi(x), \mathrm{X}^L_{H_N}(x(t))\rangle\,.
\end{split}
\end{equation}

Let focus for the moment only on the first term in the right hand side of the equation. We have
\begin{equation}\label{EQN:ConvItoSPHSFirts1}
\begin{split}
&\left . \frac{d}{dt} \right |_{t=0}  \dot{z} \langle \dd \varphi(x(t)),\mathrm{X}^L_H(x(t)) \rangle  = \ddot{z}(0) \langle \dd \varphi(x(t)), \mathrm{X}^L_H(x(t)) \rangle  +\dot{z}^(0) \langle \dd \langle \dd \varphi(x(t)), \mathrm{X}^L_H(x(t)) \rangle ,\dot{x} \rangle =\\
&=\ddot{z}(0) \langle \dd \varphi(x(t)), \mathrm{X}^L_H(x(t)) \rangle  +\dot{z}(0) \langle \dd  \langle \dd \varphi(x(t)),\mathrm{X}^L_H(x(t))\rangle ,e(x(t),z(t))\dot{z}(t) \rangle =\\
&=\ddot{z}(0) \langle \dd \varphi(x(t)),\mathrm{X}^L_H(x(t)) \rangle +\dot{z}(0) \dot{z}(0)  \dd \langle \dd \varphi(x(t)),\mathrm{X}^L_H(x(t))\rangle ,\mathrm{X}^L_H(x(t)) \rangle =\\
&= \left \langle \langle \dd \varphi(x(t)),\mathrm{X}^L_H(x(t)) L_{\ddot{z}} +  \langle \langle \dd \langle \dd \varphi(x(t)), \mathrm{X}^L_H(x(t)) \rangle,\mathrm{X}^L_H(x(t)) \rangle , L_{\ddot{z}}\right \rangle \,.
\end{split}
\end{equation}
Using now the fact that
\[
\langle \dd \varphi(x(t)), \mathrm{X}^L_H(x(t))\rangle = [\varphi,H]_L(x(t))\,,
\]
it follows from equation \eqref{EQN:ConvItoSPHSFirts1} that
\begin{equation}\label{EQN:ConvItoSPHSFirts12}
\begin{split}
&\left . \frac{d}{dt} \right |_{t=0}  \dot{z} \langle \dd \varphi(x(t)),\mathrm{X}^L_H(x(t)) \rangle  = \\
&=\left \langle [\varphi,H]_L (x(t) + [[\varphi,H]_L,H]_L (x(t)) , L_{\ddot{z}}\right \rangle\,.
\end{split}
\end{equation}

Similar computation holds for both the second and the third term in equation \eqref{EQN:ConvItoSPHS1}. Hence, evaluating  $s^*(x,z)(\dd_2 \varphi)$, for a given function $\varphi \in C^\infty(\XX)$, and exploiting equations \eqref{EQN:ConvItoSPHS1}--\eqref{EQN:ConvItoSPHSFirts1}--\eqref{EQN:ConvItoSPHSFirts12} together with condition \eqref{EQN:NullQC1}, we have
\begin{equation}\label{EQN:ItoOP}
\begin{split}
\left \langle s^*(x,z)(\dd_2 \varphi(x)),L_{\ddot{z}}\right \rangle &= \left \langle(\dd_2 \varphi(x)), s(x,z)L_{\ddot{z}}\right \rangle = s(x,z)(L_{\ddot{z}})[\varphi] = \\
&= \left \langle [\varphi,H]_L (x(t) + [[\varphi,H]_L,H]_L (x(t)) , L_{\ddot{z}}\right \rangle+\\
&+ \left \langle [\varphi,H_g]_L (x(t) u_t + [[\varphi,H_g]_Lu_t,H_g]_Lu_t (x(t)) , L_{\ddot{z}}\right \rangle+\\
&+\left \langle [\varphi,H_N]_L (x(t) + [[\varphi,H_N]_L,H_N]_L (x(t)) , L_{\ddot{z}}\right \rangle\,.
\end{split}
\end{equation}

Therefore we obtain, for any $\varphi \in C^\infty(\XX)$
\begin{equation}\label{EQN:FinalITOSPHS}
\begin{split}
d\varphi(X_t) &= \langle \dd_2 \varphi,dX_t\rangle = \left \langle s^*(X_t,\mathbf{Z}_t)(\dd_2 \varphi),d\mathbf{Z}_t \right \rangle = \\
&= [\varphi,H]_L (X_t)d Z_t + [[\varphi,H]_L,H]_L d \langle Z,Z\rangle_t+\\
&+u [\varphi,H_g]_L(X_t) d Z^g_t + u[[\varphi,H_g]_L,H_g]_L d \langle Z^g,Z^g\rangle_t+\\
&+[\varphi,H_{N}]_L (X_t)d Z^{N}_t + [[\varphi,H_N]_L,H_N]_L d \langle Z^N,Z^N\rangle_t\,,\\
\end{split}
\end{equation}
and the claim follows.
\end{proof}

\color{black}
\section{Implicit port-Hamiltonian systems}\label{SEC:IPHS}

\subsection{Implicit deterministic port-Hamiltonian systems}\label{SEC:ISPHS}

\fc{ Having provided, in Section \ref{SEC:DESPHS},} a geometric formulation of PHS, we can generalize the \fc{given definition of PHS} to introduce the notion of implicit PHS, \cite{Sec,VdSPH}. As briefly mentioned, the definition of implicit PHS is based on the notion of \textit{Dirac structure}, \fc{hence we first need} to introduce some fundamental concepts, \cite{VdSPH}.

Let $\mathcal{F}$ be a general finite dimensional linear space and $\mathcal{E}:= \mathcal{F}^*$ be its dual. The product space $\mathcal{E}\times \mathcal{F}$ is the space of power variables
\[
P := \langle e,f\rangle\,,\quad (f,e)\in \mathcal{F}\times \mathcal{E}\,;
\]
where $\langle e,f\rangle$ denotes the duality product, \fc{while} $\mathcal{F}$ is usually referred to as the space of \textit{flows} $f$, whereas $\mathcal{E}$ is the space of \textit{efforts} $e$. We can also introduce the following bilinear symmetric form
\[
\langle\langle(e_1,f_1),(e_2,f_2)\rangle\rangle := \langle e_1,f_2\rangle + \langle e_2,f_1\rangle = e_1^Tf_2 + e_2^Tf_1\,.
\]

In what follows, given a linear subspace $\mathcal{S}\subset \mathcal{E} \times \mathcal{F}$, we will define the \textit{orthogonal complement} $\mathcal{S}^{\perp}$ to be
\[
\mathcal{S}^\perp := \left \{(e,f) \in  \mathcal{E} \times \mathcal{F} \, : \, \langle\langle(e,f),(\tilde{e},\tilde{f})\rangle\rangle = 0 \,,\quad  \forall \, (\tilde{e},\tilde{f}) \in  \mathcal{E} \times \mathcal{F} \right \}\,.
\]

Therefore, we may describe a physical system as the interconnection of \textit{storage elements} $(f_S,e_S) \in  \mathcal{F}_S\times \mathcal{E}_S$, of \textit{resistive elements} $(f_R,e_R) \in  \mathcal{F}_R\times \mathcal{E}_R$ and the \textit{environment} or the \textit{control system} $(f_C,e_C) \in  \mathcal{F}_C\times \mathcal{E}_C$. In this particular case the general space of \textit{flows} is given by $\mathcal{F}:=\mathcal{F}_S \times \mathcal{F}_R \times \mathcal{F}_C$ and the space of \textit{efforts} $\mathcal{E}:=\mathcal{E}_S \times \mathcal{E}_R \times \mathcal{E}_C$.
\fc{The latter allows us to} 
introduce the notion of \textit{separable Dirac structure}, \cite[Ch. 6]{VdSBook}.

\begin{Definition}\label{DEF:CDS}
A (constant) \textit{separable Dirac structure} $\mathcal{D}$ on $\mathcal{F}$ is a linear subspace $\mathcal{D} \subset \mathcal{F} \times \mathcal{E}$ such that $\mathcal{D}=\mathcal{D}^\perp$.
\end{Definition}

Consider for the moment the particular case where the relation between resistive elements can be written in input--output form, i.e., there exists a map $F:\RR^{n_R} \to \RR^{n_R}$ such that
\begin{equation}\label{EQN:IORes}
f_R = -F(e_R)\,,\quad e^T_R F(e_R) \geq 0\,.
\end{equation}

Also, the interconnection of the energy storing elements to the storage port of the \textit{Dirac structure} is obtained setting
\begin{equation}\label{EQN:Store}
f_S = -\dot{x} \,,\quad e_S = \frac{\partial}{\partial x}H(x)\,,
\end{equation}
so that we obtain the following definition for the \textit{implicit PHS}.

\begin{Definition}[Implicit port-Hamiltonian system]\label{DEF:IPHS}
Let $\mathcal{F}$ be the space of flows and $\mathcal{E}$ its dual; let $H:\XX \to \RR$ be the \textit{Hamiltonian function} representing the energy of the system, with $\mathcal{D}$ a \textit{Dirac structure}. Then an \textit{implicit port-Hamiltonian system} is given by
\[
\left (-\dot{x}, \frac{\partial}{\partial x}H(x),-F(e_R),e_R,f_C,e_C\right ) \in \mathcal{D}\,.
\]
\end{Definition}

\subsubsection{Implicit deterministic port-Hamiltonian systems on manifolds}\label{SEC:ImplNC}

In this section we consider PHS with non--constant geometry. In order to achieve such a generalization we will consider \textit{Dirac structure} on differentiable manifolds.

Given a $n-$dimensional manifold $\XX$ with tangent bundle $\TT \XX$ and cotangent bundle $\TT^* \XX$ we will define $\TT \XX \oplus \TT^* \XX$ to be the smooth vector bundle over $\XX$ with fiber at $x \in \XX$ given by $\TT_x \XX \times \TT^*_x \XX$. We will say that $(X,\theta)$ belongs to a smooth vector subbundle $\mathcal{D} \subset \TT \XX \oplus \TT^* \XX$ if $(X(x),\theta(x)) \in \mathcal{D}(x)$, $\forall \, x \in \XX$, \fc{thereafter using} the shorthand notation $(X,\theta)\in \mathcal{D}$.

We can also introduce the \textit{orthogonal complement} w.r.t. the standard pairing between forms and vector fields as
\[
\mathcal{D}^\perp = \left \{(X,\theta) \, : \, \langle \theta,\bar{X}\rangle + \langle \bar{\theta},X\rangle=0\,, \, \forall \, (\bar{X},\bar{\theta}) \in \mathcal{D} \right \}\,,
\]
where $\langle \cdot, \cdot \rangle$ denotes the standard duality pairing between forms and vector fields as defined in equation \eqref{EQN:PairVFF1}.

Therefore we have the following definition, which generalizes Definition \ref{DEF:CDS}, \cite[Definition 2.1]{Dal}.
\begin{Definition}[Generalized Dirac structure]\label{DEF:GDS}
A generalized \textit{Dirac structure} $\mathcal{D}$ on a smooth manifold $\XX$ is a smooth vector subbundle $\mathcal{D} \subset \TT \XX \oplus \TT^* \XX$ such that $\mathcal{D}=\mathcal{D}^\perp$.
\end{Definition}

Definition \ref{DEF:GDS} implies that a generalized \textit{Dirac structure} $\mathcal{D}$ on a smooth manifold $\XX$ is a smooth vector subbundle $\mathcal{D} \subset \TT \XX \oplus \TT^* \XX$ such that $\mathcal{D}(x) \subset \TT^*_x \XX \times \TT_x \XX$ is a constant Dirac structure, in the sense of Definition \ref{DEF:CDS}, for every $x \in \XX$, see \cite[Sec. 3]{VdSPH}.

Notice that, taking $\bar{X}=X$ and $\bar{\theta}=\theta$ we immediately obtain that
\[
\langle \theta,X\rangle = 0 \,,\quad \forall \, (X,\theta) \in \mathcal{D}\,.
\]

We thus can introduce the following definition of \textit{implicit port-Hamiltonian system}, with general space of flows $\mathcal{F}$ and efforts $\mathcal{E}$.

\begin{Definition}[Implicit generalized port-Hamiltonian system]\label{DEF:IGPHS}
Let $\mathcal{F}$ be the space of flows and $\XX$ be a smooth $n$--dimensional manifold $\mathcal{X}$, $H:\mathcal{X}\to \RR$ be a Hamiltonian function  and $\mathcal{D}$ be a \textit{Dirac structure}. The \textit{implicit generalized port-Hamiltonian system} $(\XX,\mathcal{F},\mathcal{D},H)$ is defined by
\[
\left (-\dot{x},\frac{\partial H}{\partial x}(x),f,e\right ) \in \mathcal{D}(x)\,.
\] 
\end{Definition}

It can thus being shown that the explicit PHS with dissipation \eqref{EQN:EPHSDc} is a PHS as defined in Definition \ref{DEF:IGPHS}.

\begin{Proposition}\label{PRO:IsoPHS}
Consider $\mathcal{D}$ defined as
\[
\left (-X,\theta,f^R,e^R,f^C,e^C\right ) \in \mathcal{D}\,,
\]
if and only if
\begin{equation}\label{EQN:CondDSdet}
\begin{cases}
&X(x) = \left (J(x) - R(x)\right ) \theta + g(x) f^C \,,\\
&e^C = g^T(x)\theta \,,\\
\end{cases} 
\end{equation}
such that $J=-J^T$ and $R \succeq 0$, then $\mathcal{D}$ is a Dirac structure.
\end{Proposition}
\begin{proof}
Let $\left (-X,\theta,f^R,e^R,f^C,e^C\right ) \in \mathcal{D}^\perp$ we have that
\[
\begin{split}
&-\langle \bar{\theta},X\rangle - \langle \theta, \bar{X}\rangle + \langle \bar{e}^R,f^R\rangle + \langle e^R, \bar{f}^R\rangle + \langle \bar{e}^C,f^C\rangle + \langle e^C, \bar{f}^C\rangle  = 0\,, 
\end{split}
\]
for any $\left (-\bar{X},\bar{\theta},\bar{f}^R,\bar{e}^R,\bar{f}^C,\bar{e}^C\right )$ satisfying \eqref{EQN:CondDSdet}.

Choosing $\bar{f}^C = \bar{f}^R=0$, and setting
\begin{equation}\label{EQN:CondDSAdd}
\begin{cases}
e^R &= \theta\,,\\
f^R &= R(x) e^R\,,\\
f^C &= u \,\\
\end{cases} 
\end{equation}
we have that, $\forall\,\bar{\theta}$, it holds
\begin{equation}\label{EQN:IOSys12}
\begin{split}
&-\langle \bar{\theta}, X\rangle - \langle \theta,J(x)\bar{\theta} \rangle + \langle \theta,R(x) e^R\rangle + \langle  g^T(x)\theta,u \rangle  = 0\,. 
\end{split}
\end{equation}

Thus, it immediately follows, with $\theta = \partial_x H(X_t)$ and $X = \dot{x}$,
\begin{equation}\label{EQN:IOSys1}
\dot{x} = \left (J(x)-R(x)\right )\partial_x H(x) + g u \,,
\end{equation}
and inserting equation \eqref{EQN:IOSys1} into equation \eqref{EQN:IOSys12} we obtain 
\[
e^C = g^T(x)\partial_x H(x)\,,\\
\]
\fc{so that} $\left (-\dot{x},\dd H,f^R,e^R, f^C,e^C\right ) \in \mathcal{D}$.
\end{proof}

Proposition \ref{PRO:IsoPHS} motivates the following definition.

\begin{Definition}[Input--state--output port-Hamiltonian system]\label{DEF:EPHS}
Let $\XX$ be a smooth $n$--dimensional manifold, and $H:\mathcal{X}\to \RR$ be a Hamiltonian function, then
\begin{equation}\label{EQN:EPHS2}
\begin{cases}
\dot{x}=\left [J(x)-R(x)\right ]\Par{H(x)}+g(x)u\,,\\
y=g^T(x)\Par{H(x)}\,,
\end{cases}
\end{equation}
with $J(x) = -J^T(x)$ and $R(x) = R^T(x) \succeq 0$, is called \textit{input--state--output port--Hamiltonian system}.
\end{Definition}

Notice that
\begin{equation}\label{EQN:Pass}
\frac{d}{dt}H = - \Par{{}^TH(x)}R(x)\Par{H(x)}+y^Tu\leq y^Tu\,,
\end{equation}
or equivalently in integral form
\begin{equation}\label{EQN:PassInt}
H(x(t))- H(x(0)) = - \int_0^t \left (\Par{{}^TH(x)}R(x)\Par{H(x)}+y^Tu \right ) ds\leq \int_0^t y^Tu ds\,.
\end{equation}

This equation states that the internal energy of the system is always less or equal to the external energy supplied to the system.
\fc{In particular, equation \eqref{EQN:Pass} expresses what in literature is known as \textit{passivity property} of PHS, see, e.g., \cite{VdSBook}.}

\subsection{Implicit stochastic port-Hamiltonian systems}\label{SEC:ISPHS}

The main goal of the present section is to formally introduce the definition of \textit{implicit stochastic port-Hamiltonian system} (SPHS). \fc{We would like to underline that, to the best of our knowledge, no formulation of implicit SPHS has been provided in literature so far}. We will show that our definition generalizes already existing definitions of explicit input--state--output SPHS as introduced in Section \ref{SEC:PHS}, \fc{see} \cite{Had,Sat1,Sat2,Sat3,Sat4,CDPM_Bil,CDPM_IFAC,CDPMTank,CDPM_Weak}, as well as stochastic dynamics on Poisson manifolds, see \cite{Ort}.

As done in Section \ref{SEC:PHS}, we will first introduce the general notion of \textit{explicit stochastic port-Hamiltonian system} as a controlled Hamiltonian system on Poisson or Leibniz manifold. Then, inspired by the general theory of explicit SPHS on manifolds, we will generalize the theory to define \textit{implicit stochastic port-Hamiltonian system}.

This section is devoted to generalize to the stochastic setting the definition of \textit{implicit port-Hamiltonian system} of Section \ref{SEC:ImplNC}. We assume that the flow corresponding to each port is a semimartingale, so that the noise can enter into the system not only through a stochastic external random field but also as a random perturbation of any port connected to the system.

\color{black}
We use \textit{Stratonovich calculus} since it allows us to exploit standard rules of differential calculus and exterior calculus on manifold. In what follows, we will consider $X:I \to \XX$ to be an integral curve of a \textit{Stratonovich vector field} $\delta X_t$ with initial condition $X_0$, being $I \subset \RR_+$. We recall that, \cite{EmeBook}, for any differential 1-form $\theta$ on $\XX$ we can associate in a unique way the real-valued semimartingale that represents the integration of $\theta$ along the vector field $\delta X$, denoted as
\begin{equation}\label{EQN:PairVFF}
\int_0^t \langle \theta, \delta X_s\rangle\,.
\end{equation}
The integral \eqref{EQN:PairVFF} is called \textit{Stratonovich integral} of $\theta$ along $\delta X$.
\color{black}

\color{black}

We introduce the \textit{orthogonal complement} of a bundle $\mathcal{D}\subset \TT \XX \oplus \TT^* \XX$ w.r.t. the above introduced pairing between forms and vector fields as
\begin{equation}\label{EQN:OrtCom}
\begin{split}
\mathcal{D}^\perp &= \left \{ (\delta X_t,\theta)\subset \TT \XX \oplus \TT^* \XX \, : \right . \,\\
&\left . \qquad \qquad \, \int_0^t \langle \theta,\delta \bar{X}_s\rangle + \int_0^t \langle \bar{\theta},\delta X_s \rangle=0\,,  \forall \, (\delta \bar{X}_t,\bar{\theta}) \in \mathcal{D},\,t \in I \right \}\,.
\end{split}
\end{equation}

The following definition, generalizes Definition \ref{DEF:GDS}.

\begin{Definition}[Generalized stochastic Dirac structure]\label{DEF:SGDS}
A \textit{generalized stochastic Dirac structure} $\mathcal{D}$ on a manifold $\XX$ is a smooth vector subbundle $\mathcal{D} \subset \TT \XX \oplus \TT^* \XX$ such that $\mathcal{D}=\mathcal{D}^\perp$.
\end{Definition}

Notice that taking $\delta \bar{X}=\delta X$ and $\bar{\theta}=\theta$ we immediately obtain that
\begin{equation}\label{EQN:EnCons}
\int_0^t \langle \theta,\delta X_s\rangle = 0 \,,\quad \forall \, (\delta X_t,\theta) \in \mathcal{D} \,,\quad \forall t \in I \,.
\end{equation}

%Let us thus consider a $n$-dimensional manifold $\XX$ with \textit{generalized stochastic Dirac structure} $\mathcal{D}$ and let $H:\XX \to \RR$ be the \textit{Hamiltonian function} representing the energy of the system perturbed by a semimartingale $Z$. 
The following generalizes Definition \ref{DEF:IGPHS}.

\begin{Definition}[Implicit generalized stochastic port-Hamiltonian system]\label{DEF:IGSPHS}
Let $\XX$ be an $n$-dimensional manifold $\XX$ with \textit{generalized Dirac structure} $\mathcal{D}$, $H:\XX \to \RR$ the \textit{Hamiltonian function} perturbed by the semimartingale $Z$. An \textit{implicit generalized stochastic port-Hamiltonian system} $(\XX,Z,\mathcal{D},H)$ on $\XX$ is given by
\[
\left (\delta X_t,\dd H (X_t) \right ) \in \mathcal{D}(X_t)\,, \quad \forall \, t \in I\,.
\]
\end{Definition}

\fc{Next examples highlight how this definition includes main cases considered in the deterministic setting.}

\begin{Example}\label{EX:IO}
\begin{description}[style=unboxed,leftmargin=0cm]
\item[(i)] Let $\left (\XX,B\right )$ be a \textit{Poisson manifold}, with $B^{\#}: \TT^* \XX \to \TT \XX$ the \textit{Poisson morphism} introduced in Section \ref{SEC:PHS}, then
\[
\mathcal{D}_B = \left \{(\delta X,\theta )\,:\, \delta X(x)= B^{\#}\theta (x) \delta Z ,\,\theta \in T^*\XX\right \}\,,
\]
defines a Dirac structure. In particular, the defined Dirac structure leads to the \textit{Hamilton equations}
\[
\delta X_t = B^{\#}(\dd H)(X_t)\delta Z_t\,,
\]
or, equivalently, in integral form
\[
X_t =X_0+\int_0^t B^{\#}(\dd H)(X_s)\delta Z_s\,;
\]

\item[(ii)] Let $\left (\XX,\omega\right )$ be a symplectic manifold, that is $\omega$ is a closed (possibly degenerate) two--form, with $\omega^{\#} : \TT \XX \to \TT^* \XX$ the canonical musical isomorphism, then 
\[
\mathcal{D}_\omega = \left \{(\delta X,\theta)\,:\, \theta \delta Z = \omega^{\#}\left ( \delta X\right ) ,\, \delta X \in T\XX\right \}\,,
\]
is a Dirac structure.
\end{description}
\demo 
\end{Example}

Since Definition \ref{DEF:IGSPHS} is based on \fc{both} \textit{Stratonovich calculus} and \textit{exterior calculus}, it allows us to obtain the remarkable \textit{energy conservation property}, which is one of the founding aspects of port-Hamiltonian systems. In particular, we have that
\begin{equation}\label{EQN:EnConsFin}
 H(X_t)-H(X_0) = \int_0^t \langle \dd H,\delta X_s\rangle \,,
\end{equation}
or in short hand notation
\[
\delta H(X_t) =\langle \dd H ,\delta X_t\rangle\,.
\]

We can thus introduce the port variables associated with internal storage $(\delta f^S_t,e^S_t)$, \fc{interconnecting} the energy storing elements to the storage port of the Dirac structure by setting
\[
\delta f_t^S = -\delta X_t\,,\quad e_t^S=\dd H\,.
\]

The energy balance reads
\begin{equation}\label{EQN:EnBalance1}
H(X_t)-H(X_0) = \int_0^t \langle \dd H,\delta X_s\rangle = - \int_0^t \langle e_s^S,\delta f_s^S\rangle \,,
\end{equation}
and the total energy is preserved along solutions of the Hamiltonian system. We remark that
%\begin{Remark}
the particular notation $\delta f^S$ \fc{emphasizes} the fact that the flow of the storage port is a \textit{Stratonovich vector field} over $\XX$. 
\color{black}
The latter, implies one of the major novelty of the proposed approach. Since the flow variable $\delta X$ is a stochastic \textit{Stratonovich vector field}, the power $P_t$ exchange through the port
\[
P_t := \int_0^t \langle \dd H,\delta X_s\rangle\,,
\]
is a real-values semimartingale. Therefore, as previously mentioned, in the considered setting each port element can be intrinsically stochastic.

\color{black}

\begin{Remark}\label{REM:LieDer}
It is worth remarking that the Definition \ref{DEF:SGDS} of \textit{Dirac structure} has been called generalized to differentiate it to the original definition in \cite{Cou} on \textit{Dirac manifold} where a certain closeness assumption has been made. In particular, in later development of the theory, closeness assumptions was dropped, mainly with the aim of including nonholonomic constraints into the definition of Dirac structure. Using the Definition \ref{DEF:SGDS} of generalized Dirac structure we are able in Example \ref{EX:IO} to consider a (pseudo)-Poisson bracket, that is a Poisson bracket that does not satisfy Jacobi identity, and pre--symplectic geometry considering a two--form that is not necessarily closed.

\color{black}
To recover the original definition in \cite{Cou} we can require the closeness of the \textit{Dirac structure} according to the next definition.
\begin{Definition}[(Closed) Dirac structure]\label{DEF:CloDS}
A \textit{generalized Dirac structure} $\mathcal{D}$ on $\XX$ is called (closed) \textit{Dirac structure} if for arbitrary $(\delta X^1_t,\theta_1)$, $(\delta X^2_t,\theta_2)$ and $(\delta X^3_t,\theta_3)$, it holds
\[
\langle \La_{\delta X^1_t} \theta_2,\theta_3\rangle + \langle \La_{\delta X^2_t} \theta_3,\theta_1\rangle+\langle \La_{\delta X^3_t} \theta_1,\theta_2\rangle=0\,,
\]
being $\La_{\delta X_t}\theta $ the \textit{Lie-derivative} of the form $\theta$ along the \textit{Stratonovich vector field} $\delta X_t$.
\end{Definition}

It is worth stressing that, due to the geometric nature of our definitions, the \textit{Lie-derivative} of the form $\theta$ along the \textit{Stratonovich vector field} $\delta X_t$ can be defined through the \textit{Cartan magic formula}, \cite{Holm1}, as
\[
\La_{\delta X_t}\theta = \dd (\ii_{\delta X_t} \theta) + \ii_{\delta X_t} \dd \theta\,.
\]
\demo \end{Remark}

\color{black}

\subsubsection{Interconnection of the Dirac structure with other ports}\label{SEC:OP}

PHS's are mainly seen as interconnection of different port elements, possibly representing different physical systems. In the present section we will introduce the general formalism needed to connect several ports through a stochastic Dirac structure. The main idea follows what \fc{previously provided} introducing the stochastic implicit PHS in Definition \ref{DEF:IGSPHS}, and resembles how one can formally define distributed parameter PHS's, \cite{VdSC2}. In order to be able to incorporate stochasticity into the implicit stochastic PHS we will consider particular choice for effort and flow spaces.

In what follows, we will consider the flow space $\mathcal{F}_{Z^\alpha}:= \mathfrak{X}_{Z^\alpha} (\XX)$ to be the space of \textit{Stratonovich vector fields} on $\XX$ perturbed by a general semimartingale $Z^\alpha$\fc{. As to} emphasize that any flow element is in fact a \textit{Stratonovich vector field}, we will denote any element belonging to $\mathcal{F}_{Z^\alpha}$ as $\delta f^\alpha$. Similarly, we will consider the space of efforts to be the dual of the space of flows, so that $\mathcal{E}:= \Omega^1(\XX)$ is the space of $1-$forms on $\XX$. As already discussed above, to any element $(e,\delta f) \in \mathcal{E}\times \mathcal{F}_{Z^\alpha}$ we can associate a natural pairing, \cite{Holm1}. We stress that in general we can consider flow variables, resp. effort variables, to take values in the set of Stratonovich vector fields $\mathfrak{X}(\mathcal{N})$, resp. 1-forms $\Omega(\mathcal{N})$, over a different manifold $\mathcal{N}$.

\fc{Let us underline that in the implicit SPHS, it is possible to consider other ports besides energy storage ones}, such as \textit{resistive ports} $(R)$ and \textit{control ports} $(C)$. In the following we will thus denote by $\mathcal{F}_{Z^\alpha}$ the space of \textit{Stratonovich vector fields} on $\XX$ generated by a semimartingale $Z^\alpha$, $\alpha = R$, $C$, and by $\mathcal{E}_\alpha$, $\alpha=R$, $C$, be the space of 1-form on $\XX$.

\begin{Remark}
In the general implicit form, there is no need to specify the perturbing semimartingale $Z^\alpha$ for the port $\alpha$: since $\delta f^\alpha$ is a Stratonovich vector field, the whole theory would follow analogously. Nonetheless, we have chosen to specify the perturbing semimartingale also in the implicit form to emphasize the connection to explicit SPHS.

\demo
\end{Remark}

A dissipation effect can be further taken into account by terminating the resistive port with a dissipation element satisfying an energy--dissipating relation $\mathcal{R}$. In general such a relation is defined as a subset
\[
\mathcal{R} \subset \mathcal{F}_{Z_R} \times \mathcal{E}_R\,,
\]
such that it holds
\begin{equation}\label{EQN:DissRPO}
\int_0^t \langle e^R_s,\delta f^R_s \rangle \leq 0\,, \quad t \in I\,.
\end{equation}
\fc{A relevant case is the one when}  the resistive relation can be expressed as the graph of an input--output map, so that, given a map $\delta \tilde{R}:\Omega^1(\XX) \to \mathfrak{X}_{Z^R}(\XX)$, we require
\begin{equation}\label{EQN:DissRP2}
\int_0^t \langle e^R_s,\delta \tilde{R}(e^R_s) \rangle \geq 0\,,
\end{equation}
and we can impose the following connection
\[
\delta f^R_s := - \delta \tilde{R}(e^R_s)\,.
\]

\begin{Definition}[Implicit generalized stochastic port-Hamiltonian system]\label{DEF:ImpSPHS}
Let $\XX$ be an $n$--dimensional manifold $\mathcal{X}$, $\mathbf{Z} = (Z,Z^R,Z^C)$ be a semimartingale, $H:\mathcal{X}\to \RR$ be a Hamiltonian function and $\mathcal{D}$ be a \textit{generalized stochastic Dirac structure}. Let also $\mathcal{F}:= \mathcal{F}_{Z_R}\times \mathcal{F}_{Z_C}$ be the space of flows $\delta f$ and $\mathcal{E}=\mathcal{F}^*$ be the corresponding dual space of efforts. The \textit{implicit generalized port-Hamiltonian system} $(\XX,\mathbf{Z},\mathcal{F},\mathcal{D},H)$, with resistive structure $\mathcal{R}$, is defined by
\[
\left (-\delta X_t,\dd H,\delta f^R_t,e_t^R,\delta f^C_t,e^C_t\right ) \in \mathcal{D}(X_t)\,,
\] 
with the resistive relation
\[
\left (\delta f^R_t,e^R_t\right ) \in \mathcal{R}(X_t)\,.
\]
\end{Definition}

Since the resistive port is required to satisfy the dissipation relation \eqref{EQN:DissRPO}, we obtain the power balance
\[
\begin{split}
H(X_t)-H(X_0) &= \int_0^t \langle \dd H,\delta X_s\rangle = \int_0^t \langle e_s^R,\delta f_s^R\rangle + \int_0^t \langle e_s^C,\delta f_s^C\rangle \leq \\
&\leq  \int_0^t \langle e_s^C,\delta f_s^C\rangle \,.
\end{split}
\]

Notice that the condition for the resistive port \eqref{EQN:DissRPO} is usually too strong to be satisfied in practice, since it requires that energy dissipation occurs along all possible realizations of the system. In order to weaken \fc{it}, we introduce a different formulation of \textit{Dirac structure} \fc{with the weaker resistive relation of requiring that the energy being dissipated in mean value.}

The weak energy--dissipating relation $\mathcal{R}_W$ is defined as a subset
\[
\mathcal{R}_W \subset \mathcal{F}_{Z_R} \times \mathcal{E}_R\,,
\]
such that
\begin{equation}\label{EQN:DissRPW}
\mathbb{E} \int_0^t \langle e^R_s,\delta f^R_s \rangle \leq 0\,.
\end{equation}
Similarly, if there exists a map $\delta \tilde{R}:\Omega^1(\XX) \to \mathfrak{X}(\XX)$, so that
\begin{equation}\label{EQN:DissRP2W}
\mathbb{E} \int_0^t \langle e^R_s,\delta \tilde{R}(e^R_s) \rangle \geq 0\,,
\end{equation}
we can obtain energy dissipation imposing the following interconnection,
\[
\delta f^R_s := - \delta \tilde{R}(e^R_s)\,.
\]

Therefore, in the weak setting the resistive port is required to satisfy a weak dissipation condition of the form \eqref{EQN:DissRPW}, and the mean power balance reads
\[
\begin{split}
\mathbb{E} H(X_t) &= \mathbb{E} H(X_0) + \mathbb{E} \int_0^t \langle \dd H,\delta X_s\rangle =\\
&= \mathbb{E} H(X_0) + \mathbb{E} \int_0^t \langle e_s^R,\delta f_s^R\rangle + \mathbb{E} \int_0^t \langle e_s^C,\delta f_s^C\rangle \leq \\
&\leq \mathbb{E} H(X_0) + \mathbb{E} \int_0^t \langle e_s^C,\delta f_s^C\rangle \,,
\end{split}
\]
implying that energy is required to be preserved and dissipated in mean value. \fc{We stress that we will always consider the} weak relation since it is the most suitable to many applications, nonetheless similar arguments will still hold imposing strong energy dissipation relations.

\subsubsection{The general case}\label{SEC:GenPHS}

In order to generalize \textit{Hamilton equations} \eqref{EQN:ESPHSDiss} we augment the Dirac structure with a new type of port, that we will call \textit{noise port}, with flow space $\mathcal{F}_{Z_{N}}$, the space of \textit{Stratonovich vector fields} perturbed by $Z_{N}$, and effort space $\mathcal{E}_{N}$. \fc{As it will be seen later on, within the explicit formulation, this port will play the role of external random field perturbing the system.}

\begin{Definition}\label{DEF:ImpSPHSsh}
Let $\XX$ be an $n$--dimensional manifold $\mathcal{X}$, $\mathbf{Z} = (Z,Z^R,Z^C,Z^N)$ be a semimartingale, $H:\mathcal{X}\to \RR$ be a Hamiltonian function and $\mathcal{D}$ be a \textit{generalized stochastic Dirac structure}. The \textit{implicit generalized port-Hamiltonian system} $(\XX,\mathbf{Z},\mathcal{F},\mathcal{D},H)$ is defined by
\begin{equation}\label{EQN:ImpSPHSshEQ}
\left (-\delta X_t,\dd H,\delta f^R_t,e^R_t,\delta f^C_t,e_t^C,\delta f^{N}_t,e_t^{N}\right ) \in \mathcal{D}(X_t)\,.
\end{equation}
\end{Definition}

Figure \ref{FIG:Impl} shows a graphical representation of this definition.

We can also introduce the (weak) resistive relation
\[
\left (\delta f^R_t,e^R_t\right ) \in \mathcal{R}_W(X_t)\,,
\]
so that the, weak energy balance reads as
\[
\begin{split}
\mathbb{E} H(X_t) - \mathbb{E} H(X_0) &= \mathbb{E} \int_0^t \langle \dd H,\delta X_s\rangle =\\
&= \mathbb{E} \int_0^t \langle e_s^{N},\delta f_s^{N}\rangle + \mathbb{E} \int_0^t \langle e_s^R,\delta f_s^R\rangle + \mathbb{E} \int_0^t \langle e_s^C,\delta f_s^C\rangle \leq \\
&\leq \mathbb{E} \int_0^t \langle e_s^{N},\delta f_s^{N}\rangle + \mathbb{E} \int_0^t \langle e_s^C,\delta f_s^C\rangle \,,
\end{split}
\]

In many applications, we deal with passive systems, i.e. systems where the total (average) energy in the interval $[0,t]$ must be less or equal to the (average) energy injected into the system, 
\[
\mathbb{E} H(X_t) - \mathbb{E} H(X_0)\leq  \mathbb{E} \int_0^t \langle e_s^C,\delta f_s^C\rangle \,.
\]

In the deterministic case, imposing an energy dissipation relation is sufficient to guarantee the passivity of the PHS, whereas in the present case, in order to guarantee passivity, we are forced to further require the stronger condition that both the \textit{resistive port} and the \textit{noise port} satisfy a dissipativity condition. In particular, we can define an energy--dissipation relation
\[
\mathcal{R}_W^{N} \subset \mathcal{F}_{Z_R} \times  \mathcal{F}_{Z_{N}} \times \mathcal{E}_R \times \mathcal{E}_{N}\,,
\]
such that it holds
\begin{equation}\label{EQN:DissRP}
\mathbb{E}\int_0^t \langle e^R_s,\delta f^R_s \rangle + \mathbb{E} \int_0^t \langle e^{N}_s,\delta f^{N}_s \rangle \leq 0\,.
\end{equation}

Thus, endowing the stochastic PHS \eqref{EQN:ImpSPHSshEQ} with the (weak) energy--dissipation relation $\mathcal{R}_W^{N}$ we obtain the passivity property for the SPHS
\begin{equation}\label{EQN:DoubleDissStP}
\begin{split}
\mathbb{E} H(X_t) - \mathbb{E} H(X_0) &= \mathbb{E} \int_0^t \langle \dd H,\delta X_s\rangle =\\
&= \mathbb{E} \int_0^t \langle e_s^{N},\delta f_s^{N}\rangle + \mathbb{E} \int_0^t \langle e_s^R,\delta f_s^R\rangle + \mathbb{E} \int_0^t \langle e_s^C,\delta f_s^C\rangle \leq \\
&\leq \mathbb{E} \int_0^t \langle e_s^C,\delta f_s^C\rangle \,,
\end{split}
\end{equation}

As above, we can consider the situation where the general resistive relation can be expressed as the graph of an input--output map, so that, given two maps $\delta \tilde{R}:\Omega^1(\XX) \to \mathfrak{X}(\XX)$ and $\delta \tilde{R}^{N}:\Omega^1(\XX) \to \mathfrak{X}(\XX)$ we require
\begin{equation}\label{EQN:DoubleDiss}
\mathbb{E}\int_0^t \langle e^R_s,\delta \tilde{R}(e^R_s) \rangle + \mathbb{E}\int_0^t \langle e^{N}_s,\delta \tilde{R}^{N}(e^{SH}_s) \rangle \geq 0\,.
\end{equation}
By imposing the connection
\[
\begin{split}
\delta f^R_s &:= - \delta \tilde{R}(e^R_s)\,,\\
\delta f^{N}_s &:= - \delta \tilde{R}^{N}(e^{N}_s)\,,\\
\end{split}
\]
we would thus obtain the (weak) passive relation \eqref{EQN:DoubleDissStP}.

\begin{Remark}
In equation \eqref{EQN:DoubleDiss} the joint dissipativity condition for both resistive and stochastic ports is more general than requiring that dissipativity holds for both ports separately. In fact, many concrete applications satisfy a dissipativity condition for the resistive port, at least in the weak setting. Nonetheless, it is much harder to find applications where also the stochastic port does satisfy a similar dissipativity condition, even if required to hold just in weak form. Nonetheless equation \eqref{EQN:DoubleDiss} is more general since the dissipativity of the resistive port can "absorb" non--passive behaviours at the stochastic port so that the whole system remains passive; a similar reasoning has been used in \cite{CDPMTank} to define stochastic energy tanks in a bilateral teleoperation setting.

\demo
\end{Remark}

\begin{figure}
\centering
\resizebox {0.7\textwidth} {!} {
\begin{tikzpicture}[
  font=\sffamily,
  every matrix/.style={ampersand replacement=\&,column sep=2cm,row sep=2cm},
  source/.style={draw,thick,circle,fill=yellow!15,inner sep=.3cm ,minimum size=40mm},
  process/.style={draw,thick,circle,fill=blue!15,minimum  size=40mm},
  sink/.style={source,fill=green!20},
  datastore/.style={draw,very thick,shape=datastore,inner sep=.3cm},
  dots/.style={gray,scale=2},
  to/.style={->,>=stealth',shorten >=1pt,semithick,font=\sffamily\footnotesize},
  every node/.style={align=center}]

  % Position the nodes using a matrix layout
  \matrix{\& \node[process] (r) {Resistive \\ element \\ (R)}; \& \\

\node[process] (s) {Energy--storing \\ element \\(S)};   \& \node[source] (d) {Power--conserving \\ interconnection\\\\ $\Scale[2]{\mathcal{D}}$}; \&
 \node[process] (hs) {External element --\\ Noise\\(N)};\\

      \& \node[process] (c) {External element --\\ control\\(C)};
      \& ; \\
  };
  
  \draw[transform canvas={xshift=-7pt}] (r) -- node[midway,left] {\scalebox{1}{}} (d);
  \draw[transform canvas={xshift=7pt}] (r) -- node[midway,right] {\scalebox{1}{ports}} (d);
  \draw[transform canvas={yshift=-7pt}] (s) -- node[midway,below] {\scalebox{1}{}} (d);
  \draw[transform canvas={yshift=7pt}] (s) -- node[midway,above] {\scalebox{1}{ports}} (d);
  \draw[transform canvas={xshift=7pt}] (c) -- node[midway,right] {\scalebox{1}{}} (d);
  \draw[transform canvas={xshift=-7pt}] (c) -- node[midway,left] {\scalebox{1}{ports}} (d);
  \draw[transform canvas={yshift=-7pt}] (hs) -- node[midway,below] {\scalebox{1}{}} (d);
  \draw[transform canvas={yshift=7pt}] (hs) -- node[midway,above] {\scalebox{1}{ports}} (d);
\end{tikzpicture}
}
\caption{Implicit port-Hamiltonian system with dissipation, external control and interaction Hamiltonian.}\label{FIG:Impl}
\end{figure}
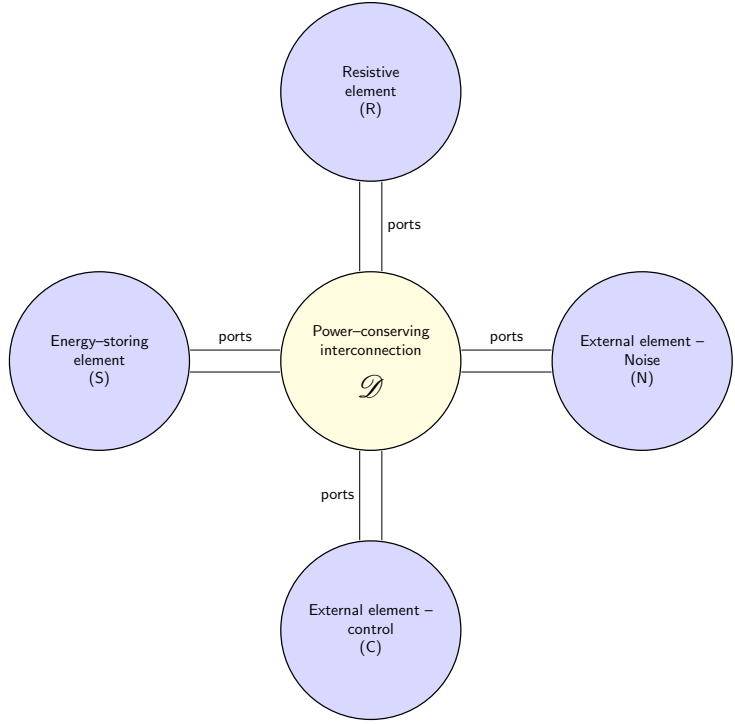

The next proposition gives an alternative representation for the Dirac structure.

\begin{Proposition}\label{PRO:KernRepSDS1}
Let $\mathcal{F}:=\mathcal{F}_{Z_R}\times \mathcal{F}_{Z_C}\times \mathcal{F}_{Z_{N}}$ be the space of flows $\delta f$ and $\mathcal{E}=\mathcal{F}^*$ be the corresponding dual space of efforts $e$, set
\begin{equation}\label{EQN:DSExpl}
\begin{split}
\mathcal{D} &:= \{ (\delta f^S_t, \delta f^R_t,\delta f^C_t,\delta f^{N}_t,e^S_t,e^R_t,e^C_t,e^{N}_t) \in \mathcal{F}\times \mathcal{E} \, : \\
& \qquad \, \,\delta f_t^S = -J e^S_t \delta Z_t - G_R \delta f^R_t - G_C \delta f^C_t - G_{N} \delta f^{N}_t\,,\\
&\qquad \,\, e^R_t = G^*_R e^S_t \,,\quad e^C_t = G^*_C e^S_t\,,\quad e^{N}_t = G^*_{N} e^S_t \}\,,
\end{split}
\end{equation}
where $G_\theta : \mathcal{F}_{Z_\theta} \to \mathcal{F}_{Z_\theta}$, $\theta =R$, $C$, $N$, such that 
\[
\langle e^S_t,G_\theta \delta f^\theta_t\rangle = \langle G_\theta^* e^S_t,\delta f^\theta_t\rangle \,, 
\]
and $J$ such that $J=-J^T$; then $\mathcal{D}$ is a Dirac structure.
\end{Proposition}
\begin{proof}
Let us first prove $\mathcal{D} \subset \mathcal{D}^\perp$. For the sake of brevity we will prove the case with $G_C = G_{N}=0$, \fc{the general case being analogous}. Let $(\delta f^S_t, \delta f^R_t,e^S_t,e^R_t)$ and $(\delta \bar{f}^S_t, \delta \bar{f}^R_t,\bar{e}^S_t,\bar{e}^R_t) \in \mathcal{D}$ be as defined in equation \eqref{EQN:DSExpl}; since they belong to $\mathcal{D}$ it holds
\begin{equation}\label{EQN:RepI}
\begin{split}
&\int_0^t \langle e^S_s,\delta \bar{f}^S_s\rangle + \int_0^t \langle \bar{e}^S_s,\delta f^S_s\rangle + \int_0^t \langle e^R_s,\delta \bar{f}^R_s\rangle + \int_0^t \langle \bar{e}^R_s,\delta f^R_s\rangle = \\
=&-\int_0^t \langle e^S_s, J\bar{e}^S_s \delta Z_s\rangle - \int_0^t \langle e^S_s, G_R \delta \bar{f}^R_s\rangle - \int_0^t \langle \bar{e}^S_s,J e^S_s \delta Z_s\rangle - \int_0^t \langle \bar{e}^S_s, G_R \delta f^R_s\rangle+\\
+& \int_0^t \langle e^R_s,\delta \bar{f}^R_s\rangle + \int_0^t \langle \bar{e}^R_s,\delta f^R_s\rangle \,.
\end{split}
\end{equation}
Thus we have that 
\[
\int_0^t \langle e^S_s, G_R \delta \bar{f}^R_s\rangle = \int_0^t \langle G_R^* e^S_s, \delta \bar{f}^R_s\rangle = \int_0^t \langle e^R_s, \delta \bar{f}^R_s \rangle \,,
\]
so \eqref{EQN:RepI} becomes, using the skew-symmetry of $J$,
\begin{equation}\label{EQN:RepII}
\begin{split}
&\int_0^t \langle e^S_s,\delta \bar{f}^S_s\rangle + \int_0^t \langle \bar{e}^S_s,\delta f^S_s\rangle + \int_0^t \langle e^R_s,\delta \bar{f}^R_s\rangle + \int_0^t \langle \bar{e}^R_s,\delta f^R_s\rangle = \\
&=-\int_0^t \langle e^S_s, J\bar{e}^S_s \delta Z_s\rangle - \int_0^t \langle \bar{e}^S_s,J e^S_s \delta Z_s\rangle =0 \,,
\end{split}
\end{equation}
and thus $(\delta f^S_t, \delta f^R_t,e^S_t,e^R_t) \in \mathcal{D}^\perp$.

Let us then prove that $\mathcal{D}^\perp \subset \mathcal{D}$; let $(\delta f^S_t, \delta f^R_t,e^S_t,e^R_t) \in \mathcal{D}^\perp$, then for all $(\delta \bar{f}^S_t, \delta \bar{f}^R_t,\bar{e}^S_t,\bar{e}^R_t) \in \mathcal{D}$ it holds
\begin{equation}\label{EQN:RepIConv}
\begin{split}
0=&\int_0^t \langle e^S_s,\delta \bar{f}^S_s\rangle + \int_0^t \langle \bar{e}^S_s,\delta f^S_s\rangle + \int_0^t \langle e^R_s,\delta \bar{f}^R_s\rangle + \int_0^t \langle \bar{e}^R_s,\delta f^R_s\rangle = \\
=&-\int_0^t \langle e^S_s, J\bar{e}^S_s \delta Z_s\rangle - \int_0^t \langle e^S_s, G_R \delta \bar{f}^R_s\rangle+\int_0^t \langle \bar{e}^S_s,\delta f^S_s\rangle+\\
&+ \int_0^t \langle e^R_s,\delta \bar{f}^R_s\rangle + \int_0^t \langle \bar{e}^R_s,\delta f^R_s\rangle \,.
\end{split}
\end{equation}
Choosing $\bar{e}^S_s=0$, $\bar{e}^R_s=0$, it follows that
\begin{equation}\label{EQN:RepIIConv}
\begin{split}
0&= - \int_0^t \langle e^S_s, G_R \delta \bar{f}^R_s\rangle+ \int_0^t \langle e^R_s,\delta \bar{f}^R_s\rangle = \\
&=- \int_0^t \langle  G_R^* e^S_s, \delta \bar{f}^R_s\rangle+ \int_0^t \langle e^R_s,\delta \bar{f}^R_s\rangle = \int_0^t \langle  e^R_s - G_R^* e^S_s, \delta \bar{f}^R_s\rangle\,,
\end{split}
\end{equation}
and from the non-degeneracy we get $e^R_s= G_R^* e^S_s$. 

\fc{Still exploiting \eqref{EQN:RepIConv}, choosing  $\delta \bar{f}^R_s=0$ and since $(\delta \bar{f}^S_t, \delta \bar{f}^R_t,\bar{e}^S_t,\bar{e}^R_t) \in \mathcal{D}$,} it follows
\begin{equation}\label{EQN:RepIIIConv}
\begin{split}
0&=-\int_0^t \langle e^S_s, J\bar{e}^S_s \delta Z_s\rangle+\int_0^t \langle \bar{e}^S_s,\delta f^S_s\rangle+ \int_0^t \langle \bar{e}^R_s,\delta f^R_s\rangle = \\
&=\int_0^t \langle \bar{e}^S_s,\delta f^S_s\rangle + \int_0^t \langle \bar{e}^S_s,Je^S_s \delta Z_s\rangle+\int_0^t \langle \bar{e}^S_s,G^*_R \delta f^S_s\rangle = \\
&=\int_0^t \langle \bar{e}^S_s,\delta f^S_s + Je^S_s \delta Z_s + G^*_R \delta f^S_s \rangle \,,
\end{split}
\end{equation}
where the last term in the second equality in equation \eqref{EQN:RepIIIConv} follows from the fact that since $(\delta \bar{f}^S_t, \delta \bar{f}^R_t,\bar{e}^S_t,\bar{e}^R_t) \in \mathcal{D}$ it holds
\[
\bar{e}^R_s = G^*_R \bar{e}^S_s\,;
\]
\fc{ so that, again by non--degeneracy} we end up with
\[
\delta f^S_s = - Je^S_s \delta Z_s - G^*_R \delta f^S_s\,,
\]
and the proof is thus complete.
\end{proof}

Let us consider the particular case with
\[
\delta f_t^R = - \tilde{R} e^R_t \delta Z_t\,,\quad \delta f^{N}_t = \xi_t \delta Z^{N}_t \,,\quad \delta f^{C}_t = u_t \delta Z^C_t\,,
\]
such that
\begin{equation}\label{EQN:DissCond}
\mathbb{E} \int_0^t \langle e^R_s , \tilde{R} e^R_s \delta Z_s \rangle - \mathbb{E} \int_0^t \langle e^{N}_s, f^{N}_s \delta Z^{N}_s \rangle \geq 0\,,
\end{equation}
then the following definition can be given.

\begin{Definition}[Stochastic input--state--output port-Hamiltonian system]\label{DEF:SPHSExplicitI}
Let $\XX$ be an $n$--dimensional manifold $\mathcal{X}$, $\mathbf{Z} = (Z,Z^R,Z^C,Z^N)$ be a semimartingale, $H:\mathcal{X}\to \RR$ be a Hamiltonian function and $\mathcal{D}$ be a \textit{generalized stochastic Dirac structure}. The \textit{stochastic input--output port-Hamiltonian system} with \textit{stochastic Dirac structure} in equation \eqref{EQN:DSExpl} is given by
\begin{equation}\label{EQN:INSPHSFin1}
\begin{cases}
\delta X_t &= \left (\tilde{J} + G_R \tilde{R} G^*_R\right ) \dd H(X_t) \delta Z_t - G_C u_t \delta Z^C_t- G_{N} \xi_t \delta Z^{N}_t\,,\\
e^{N}_t &= G_{N}^* \dd H(X_t)\,,\\
e^C_t &= G^*_C \dd H(X_t)\,.\\
\end{cases}
\end{equation}
\end{Definition}

According to \cite{EmeBook}, we have denoted the \textit{Stratonotich operator} as
\[
e(x,z) : T_z \RR^m \to T_x \XX\,;
\]
by identifying $T_z \RR^m \simeq \RR^m$, for a given $\RR^m-$ valued semimartingale $\mathbf{Z}$, we can define the $\XX-$valued SDE as
\[
\delta X_t = e(X_t,\mathbf{Z}_t)\mathbf{Z}_t\,,\quad t \in I\,.
\]

\color{black}
Equation \eqref{EQN:INSPHSFin1} can be rewritten in term of \textit{Stratonovich operator}. Consider $Z$ to be a $\RR-$valued semimartingale, whereas $Z^{N}_t$, resp. $Z^{C}_t$, is a $\RR^{n^{N}}-$valued, resp. $\RR^{n^{C}}-$valued, semimartingale, with $m=1 + n^{N} + n^C$. Denote for short the vector fields
\begin{equation}\label{EQN:VFIO}
\begin{split}
&\left (\tilde{J} + G_R \tilde{R} G^*_R\right ) \dd H =\mathrm{V}^S\,,\\
G_{N}\xi_t &= \sum_{i=1}^{n^{N}} \mathrm{V}^{N}_i \,,\quad G_{C} = \sum_{i=1}^{n^{N}} \mathrm{V}^{C}_i \,,
\end{split}
\end{equation} 
where $\mathrm{V}^\alpha_j$, $\alpha = S,\,N,\,C$, are vector fields over $\XX$.
\color{black}
Let $\{e_1^\alpha,\dots,e^\alpha_{n^\alpha}\}$ be a basis for $\RR^{n^\alpha}$, $\alpha=N,\,C$; for $y = (y^S,y^{N},y^C) \in \RR \times \RR^{n^{N}} \times \RR^{n^C}$, we can define the \textit{Stratonovich operator} as
\begin{equation}\label{EQN:StratOP}
\begin{split}
e(x,z)(y^S,y^{N},y^C) &= y^S \mathrm{V}^S(x) - \sum_{i=1}^{n^{N}} y^{N}_i \mathrm{V}^{N}_i(x) - \sum_{i=1}^{n^{C}} y^C_i \mathrm{V}^{C}_i(x) u^i_t =
\\
&= e^S(x,z)(y^S) - e^{N}(x,z)(y^{N}) - e^C(x,z)(y^C)\,,
\end{split}
\end{equation}
so that equation \eqref{EQN:INSPHSFin1} can be formally defined as a Stratonovich SDE over the manifold $\XX$ as
\begin{equation}\label{EQN:INSPHSFin1SO}
\delta X_t = e^S(X_t,Z_t)\delta Z_t - e^{N}(X_t,Z^{N}_t)\delta Z^{N}_t -e^C(X_t,Z^C_t)\delta Z^C_t\,.  
\end{equation}

Notice that, equation \eqref{EQN:INSPHSFin1SO} implies that, for any $\varphi \in C^\infty(\XX)$,
\begin{equation}\label{EQN:fStratLong}
\begin{split}
\varphi(X_t) -\varphi(X_0) = \int_0^t (\mathrm{V}^S \varphi)(X_s)\delta Z_s &- \sum_{i=1}^{n^{N}} \int_0^t (\mathrm{V}^{N}_i \varphi)(X_s) \delta Z^{N;i}_s +\\
&- \sum_{i=1}^{n^{C}} \int_0^t (\mathrm{V}^{C}_i u^i_t \varphi)(X_s) \delta Z^{C;i}_s \,.
\end{split}
\end{equation}

\begin{Example}\label{EX:IO2}
\begin{description}[style=unboxed,leftmargin=0cm]
\item[(i)]
As in Example \ref{EX:IO} let $\left (\XX,B\right )$ be a \textit{Poisson manifold}, with $B^{\#}: \TT^* \XX \to \TT \XX$ the \textit{Poisson morphism} introduced in Section \ref{SEC:PHS}, then
\[
\mathcal{D}_B = \left \{(B^{\#}\theta,\theta)\,:\,\theta \in T^*\XX\right \}
\]
defines a Dirac structure which leads to the \textit{Hamilton equations}
\[
\delta X_t = B^{\#}(\dd H)(X_t)\delta Z_t + B^{\#}(\dd H_{N})(X_t)\delta Z^{N}_t+ u B^{\#}(\dd H_C)(X_t)\delta Z^C_t\,.
\]
\fc{Notice that, in the autonomous case, namely when $u \equiv 0$, previous equation coincides with Hamilton dynamics on Poisson manifold as in \cite{Bis,Ort}.}

\item[(ii)] Let $\left (\XX,B\right )$ be a \textit{Leibniz manifold}, with $B^{\#}: \TT^* \XX \to \TT \XX$ the associated morphism as defined in equation \eqref{EQN:LeiMor23}, then
\[
\mathcal{D}_B = \left \{(B^{\#}_L\theta,\theta)\,:\,\theta \in T^*\XX\right \}
\]
defines a Dirac structure \fc{leading to}
\begin{equation}\label{EQN:SPHSDissRap}
\begin{split}
\delta X_t &= B^{\#}_L(\dd H)(X_t)\delta Z_t + B^{\#}_L(\dd H_{N})(X_t)\delta Z^{N}_t + u B^{\#}_L(\dd H_C)(X_t)\delta Z^C_t\,,
\end{split}
\end{equation}
so that stochastic Hamilton dynamics introduced in equation \eqref{EQN:ESPHSDiss} can be framed within the SPHS setting as well.
\end{description}
\demo 
\end{Example}

\fc{Therefore, we can} prove that Definition \ref{DEF:SPHSExplicitI} generalizes the classical deterministic PHS given in equation \eqref{EQN:EPHSDc}.

\begin{Proposition}\label{PRO:LocVSPHS}
\fc{Consider $\mathcal{D}$ defined as}
\[
\left (-\delta X_t,\dd H,\delta f^R_t,e^R,\delta f^C_t,e^C,\delta f^{N}_t,e^{N}\right ) \in \mathcal{D}(X_t)\,,
\]
if and only if
\begin{equation}\label{EQN:CondDS}
\begin{cases}
&\delta X_t = \left (J(X_t) - R(X_t)\right ) \partial_x H(X_t) \delta Z_t + g(X_t) u \delta Z^C_t + \xi(X_t) \delta Z^{N}_t  \,,\\
&e^C = g^T(X_t)\partial_x H(X_t) \,,\\
%&e^{N} =  \xi^T(X_t)\partial_x H(X_t) \,,
\end{cases} 
\end{equation}
with $J=-J^T$, then $\mathcal{D}$ defines a Dirac structure.
\end{Proposition}
\begin{proof}
Consider $\left (-\delta X_t,\theta,\delta f^R_t,e^R,\delta f^C_t,e^C,\delta f^{N}_t,e^{N}\right ) \in \mathcal{D}^\perp$, we have that
\[
\begin{split}
&-\langle \bar{\theta},\delta X_t\rangle - \langle \theta,\delta \bar{X}_t\rangle + \langle \bar{e}^R,\delta f^R_t\rangle + \langle e^R,\delta \bar{f}^R_t\rangle +\\
&+ \langle \bar{e}^C,\delta f^C_t\rangle + \langle e^C,\delta \bar{f}^C_t\rangle + \langle \bar{e}^{N},\delta f^{N}_t\rangle + \langle e^{N},\delta \bar{f}^{N}_t\rangle = 0\,, 
\end{split}
\]
for any $\left (-\delta \bar{X}_t,\bar{\theta},\delta \bar{f}^R_t,\bar{e}^R,\delta \bar{f}^C_t,\bar{e}^C,\delta f^{N}_t,e^{N}\right )$ satisfying \eqref{EQN:CondDS}.

If $\delta \bar{f}^C_t = \delta \bar{f}^R_t=\delta \bar{f}^{N}_t = 0$, we get
\begin{equation}\label{EQN:CondDSAdd}
\begin{cases}
e^R &= \theta\,,\\
\delta f_t^R &= R(X_t) e^R \delta Z_t\,,\\
\delta f^C_t &= u_t \delta Z^C_t\,\\
\delta f^{N}_t &= \xi_t \delta Z^{N}_t\,\\
\end{cases} 
\end{equation}
that, $\forall \bar{\theta}$, implies
\begin{equation}\label{EQN:IOSys1}
\begin{split}
&-\langle \bar{\theta},\delta X_t\rangle - \langle \theta,J(X_t)\bar{\theta} \delta Z_t\rangle + \langle \bar{\theta},R(x) e^R \delta Z_t\rangle +\\
&+ \langle  g^T(X_t)\bar{\theta},u_t \delta Z^C_t\rangle + \langle  \xi^T_{N}(X_t)\bar{\theta},\xi(X_t)\delta Z^{N}_t\rangle = 0\,. 
\end{split}
\end{equation}
Thus it immediately follows, with $\theta = \partial_x H(X_t)$,
\begin{equation}\label{EQN:IOSys}
\delta X_t = \left (J(X_t)-R(X_t)\right )\partial_x H(X_t) \delta Z_t + g(X_t) u_t \delta Z^C_t+\xi(X_t)\delta Z^{N}_t\,,
\end{equation}
and inserting equation \eqref{EQN:IOSys} into equation \eqref{EQN:IOSys1} we obtain 
\[
\begin{cases}
e^R = \partial_x H(X_t)\,\\
e^C = g^T(x)\partial_x H(X_t)\,,\\
%e^{N}=g^T_{H}(x)\partial_x H(X_t)\,,
\end{cases}
\]
and thus $\left (-\delta X_t,\dd H,\delta f^R_t,e^R,\delta f^C_t,e^C,\delta f^{N}_t,e^{N}\right ) \in \mathcal{D}$.
\end{proof}

\subsubsection{On It\^{o} definition for implicit stochastic port--Hamiltonian systems}\label{SSEC:Ito}

All results regarding implicit SPHS \fc{can be introduced} exploiting the notion of tangent and cotangent bundle of order 2.  \fc{The latter implies that the corresponding implicit stochastic Hamiltonian system is defined in It\^{o} sense}. In particular, Definition \ref{DEF:GDS} can be directly generalized to consider It\^{o} stochastic vector fields as follows.

\begin{Definition}[Generalized stochastic Dirac structure of order 2]\label{DEF:SGDS2}
A \textit{generalized stochastic Dirac structure} of order 2, $\mathcal{D}_2$, on a manifold $\XX$ is a smooth vector subbundle $\mathcal{D} \subset \tau \XX \oplus \tau^* \XX$ such that $\mathcal{D}_2=\mathcal{D}^\perp_2$, being $\mathcal{D}^\perp_2$ the orthogonal complement defined as
\[
\begin{split}
\mathcal{D}^\perp_2 &= \left \{ (d X_t,\theta)\subset \tau \XX \oplus \tau^* \XX \, : \right . \,\\
&\left . \qquad \qquad \, \int_0^t \langle \theta,d \bar{X}_s\rangle + \int_0^t \langle \bar{\theta},d X_s \rangle=0\,,  \forall \, (d \bar{X}_t,\bar{\theta}) \in \mathcal{D}_2,\,t \in I \right \}\,.
\end{split}
\]
\end{Definition}

Then, exactly as \fc{in Section} \ref{SEC:ISPHS}, we can connect different ports in a power preserving manner using It\^{o} stochastic vector fields. In what follows, with a slight abuse of notation, we will denote for short by $\mathcal{F}_{Z^\alpha}:= \mathfrak{X}_2^{Z^\alpha} (\XX)$ the space of \textit{It\^{o} vector fields} on $\XX$ perturbed by a general semimartingale $Z^\alpha$; to emphasize that any flow element is a \textit{It\^{o} vector field}, we will denote any element belonging to $\mathcal{F}_{Z^\alpha}$ as $d f$. Consequently $\mathcal{E}:= \Omega_2(\XX)$ is the space of form of order 2 on $\XX$. Thus, to any element $(e,d f) \in \mathcal{E}\times \mathcal{F}_{Z^\alpha}$ we can associate a natural pairing, \cite{Holm1}.

\begin{Definition}\label{DEF:ImpSPHSsh2}
Let $\XX$ be an $n$--dimensional manifold $\mathcal{X}$, $\mathbf{Z} = (Z,Z^R,Z^C,Z^N)$ be a semimartingale, $H:\mathcal{X}\to \RR$ be a Hamiltonian function and $\mathcal{D}_2$ be a \textit{generalized stochastic Dirac structure} of order 2. The \textit{implicit generalized port-Hamiltonian system} $(\XX,\mathbf{Z},\mathcal{F},\mathcal{D}_2,H)$ is defined by
\[
\left (-d X_t,\dd_2 H,d f^R_t,e^R_t,d f^C_t,e_t^C,d f^{N}_t,e_t^{N}\right ) \in \mathcal{D}_2(X_t)\,.
\]
\end{Definition}

This Definition can be endowed with (weak) resistive relation (as in Definition \ref{DEF:ImpSPHSsh}) of the form
\[
\left (d f^R_t,e^R_t,d f^{N}_t,e_t^{N}\right ) \in \mathcal{R}_W(X_t)\,,
\]
requiring that
\[
\mathbb{E}\int_0^t \langle e^R_s,d f^R_s \rangle + \mathbb{E} \int_0^t \langle e^{N}_s,d f^{N}_s \rangle \leq 0\,.
\]

Same computation as in Proposition \ref{PRO:KernRepSDS1} \fc{allows us to obtain} a input--state--output form as in Definition \ref{DEF:SPHSExplicitI}. \fc{We will omit details, for the sake of brevity, while showing how one can reformulate SPHS in Definition \ref{DEF:SPHSExplicitI} in It\^{o} form.}

\begin{Proposition}\label{PRO:SPHSTOIto}
Let $X$ be the solution to the Stratonovich SPHS \eqref{EQN:INSPHSFin1}, and let $Z$, $Z^{N}$ and $Z^C$ be such that
\begin{equation}\label{EQN:NullQC}
\langle Z,Z^C\rangle_t = \langle Z,Z^{N}\rangle_t = \langle Z^{N},Z^C\rangle_t = 0\,,
\end{equation}
where $\langle Z^j,Z^i\rangle_t$ is the quadratic covariation between $Z^i$ and $Z^j$ at time $t$. Then $X$ admits an equivalent formulation in terms of  It\^{o} integration 
\begin{equation}\label{EQN:ItoSPHSKer}
\begin{split}
dX_t &= \mathrm{V}^S(X_t) dZ_t +  \La_{\mathrm{V}^S} \mathrm{V}^S(X_t) d\langle  Z,Z \rangle_t+\\
&-\sum_{i= 1}^{n^{N}}\mathrm{V}^{N}_i (X_t) dZ^{N}_t- \frac{1}{2} \sum_{i,j = 1}^{n^{N}} \La_{\mathrm{V}^{N}_j} \mathrm{V}^{N}_i (X_t) d\langle Z^{N;i},Z^{N;j}\rangle_t+\\
&-\sum_{i= 1}^{n^C} \mathrm{V}_i^C  (X_t)u_t^i dZ^{C;i}_t - \frac{1}{2} \sum_{i,j = 1}^{n^C} \La_{\mathrm{V}^C_j} \mathrm{V}^C_i (X_t)u_t d\langle Z^{C;i},Z^{C;j}\rangle_t\,.
\end{split}
\end{equation}
\end{Proposition}
\begin{Remark}
The condition \eqref{EQN:NullQC} is {\fc considered only to avoid heavy notation, a similar result being true also dropping it.}

\demo
\end{Remark}
\begin{proof}
Recall that equation \eqref{EQN:INSPHSFin1} is formulated in terms of the Stratonovich operator using \eqref{EQN:StratOP} and for any $\varphi \in C^\infty (\XX)$ equation \eqref{EQN:fStratLong} holds.

Let us consider a second order vector $L_{\ddot{v}}\in \RR^{m}$, we thus have
\begin{equation}\label{EQN:ConvItoSPHS}
\begin{split}
s(x,z)(L_{\ddot{z}})[\varphi] &= \left . \frac{d}{dt} \right |_{t=0} \langle \dd \varphi(x(t)), \dot{x}(t) \rangle =  \left . \frac{d}{dt} \right |_{t=0} \langle \dd \varphi(x(t)), e(x(t),z(t)) \dot{z}(t) \rangle = \\
&=\left . \frac{d}{dt} \right |_{t=0} \dot{z}^S\langle \dd \varphi(x(t)), \mathrm{V}^S \rangle +\\
&-\left . \frac{d}{dt} \right |_{t=0} \sum_{i=1}^{n^{N}} \dot{z}^{N;i} \langle \dd \varphi(x(t)), \mathrm{V}^{N}_i \rangle +\\
&- \left . \frac{d}{dt} \right |_{t=0} \sum_{i=1}^{n^{C}} \dot{z}^{C;i} \langle  \dd \varphi(x), \mathrm{V}^C_i u_t^i \rangle\,.
\end{split}
\end{equation}

The first term in the right hand side of equation \eqref{EQN:ConvItoSPHS}, can be rewritten as
\begin{equation}\label{EQN:ConvItoSPHSFirts}
\begin{split}
&\left . \frac{d}{dt} \right |_{t=0}  \dot{z}^S \langle \dd \varphi(x(t)),\mathrm{V}^S(x(t)) \rangle  =\\
&=\ddot{z}^S(0) \langle \dd \varphi(x(t)), \mathrm{V}^S(x(t)) \rangle  +\dot{z}^S(0) \langle \dd \langle \dd \varphi(x(t)), \mathrm{V}^S(x(t)) \rangle ,\dot{x} \rangle =\\
&=\ddot{z}^S(0) \langle \dd \varphi(x(t)), \mathrm{V}^S(x(t)) \rangle  +\dot{z}^S(0) \langle \dd  \langle \dd \varphi(x(t)), \mathrm{V}^S(x(t)) \rangle ,e(x(t),z(t))\dot{z}(t) \rangle =\\
&=\ddot{z}^S(0) \langle \dd \varphi(x(t)), \mathrm{V}^S(x(t)) \rangle +\dot{z}^S(0) \dot{z}^S(0)  \dd \langle \dd \varphi(x(t)),\mathrm{V}^S (x(t))\rangle ,\mathrm{V}^S(x(t)) \rangle =\\
&= \left \langle \langle \dd \varphi(x(t)),\mathrm{V}^S(x(t)) L_{\ddot{z}} +  \langle \langle \dd \langle \dd \varphi(x(t)), \mathrm{V}^S(x(t)) \rangle,\mathrm{V}^S (x(t)) \rangle , L_{\ddot{z}}\right \rangle\,.
\end{split}
\end{equation}

Similar computation holds for all the other terms in equation \eqref{EQN:ConvItoSPHS}. \fc{Therefore, evaluating  $s^*(x,z)(\dd_2 \varphi)$, for a given function $\varphi \in C^\infty(\XX)$, we have, exploiting equations \eqref{EQN:ConvItoSPHS}--\eqref{EQN:ConvItoSPHSFirts} together with condition \eqref{EQN:NullQC}, that}
\begin{equation}\label{EQN:ItoOP}
\begin{split}
\left \langle s^*(x,z)(\dd_2 \varphi(x)),L_{\ddot{z}}\right \rangle &= \left \langle(\dd_2 \varphi(x)), s(x,z)L_{\ddot{z}}\right \rangle = s(x,z)(L_{\ddot{z}})[\varphi] = \\
&= \left \langle \langle \dd \varphi,\mathrm{V}^S \rangle , L_{\ddot{z}}\right \rangle+ \left \langle \langle \dd \langle \dd \varphi, \mathrm{V}^S \rangle ,\mathrm{V}^S \rangle, L_{\ddot{z}}\right \rangle+\\
&-\left \langle \sum_{i,j = 1}^{n^{N}} \langle \dd \varphi, \mathrm{V}_i^{N} \rangle , L_{\ddot{z}}\right \rangle+\\
&- \left \langle \sum_{i,j = 1}^{n^{N}} \langle \dd \langle \dd \varphi, \mathrm{V}_i^{N} \rangle, \mathrm{V}_j^{N} \rangle, L_{\ddot{z}}\right \rangle+\\
&- \frac{1}{2}\left \langle \sum_{i,j = 1}^{n^C} \langle \dd \varphi,\mathrm{V}_i^C u_t^i \rangle , L_{\ddot{z}}\right \rangle +\\
&- \frac{1}{2}\left \langle \sum_{i,j = 1}^{n^C} \langle \dd \langle \dd \varphi,\mathrm{V}_i^C u_t^i\rangle ,\mathrm{V}_j^C u_t^j \rangle , L_{\ddot{z}}\right \rangle \,. 
\end{split}
\end{equation}

\fc{Hence, for any $\varphi \in C^\infty(\XX)$, we obtain}
\begin{equation}\label{EQN:FinalITOSPHS}
\begin{split}
d\varphi(X_t) &= \langle \dd_2 \varphi,dX_t\rangle = \left \langle s^*(X_t,\mathbf{Z}_t)(\dd_2 \varphi),d\mathbf{Z}_t \right \rangle = \\
&=\langle \dd \varphi,\mathrm{V}^S \rangle (X_t) dZ_t + \langle \dd \langle \dd \varphi, \mathrm{V}^S \rangle ,\mathrm{V}^S \rangle (X_t) d\langle  Z,Z \rangle_t+\\
&-\sum_{i= 1}^{n^{N}} \langle \dd \varphi,\mathrm{V}^{N}_i \rangle (X_t) dZ^{N}_t- \frac{1}{2} \sum_{i,j = 1}^{n^{N}} \langle \dd  \langle \dd \varphi, \mathrm{V}_j^{N} \rangle, \mathrm{V}_i^{N} \rangle (X_t) d\langle Z^{N;i},Z^{N;j}\rangle_t+\\
&-\sum_{i= 1}^{n^C} \langle \dd \varphi,\mathrm{V}_i^C \rangle (X_t)u_t^i dZ^{C;i}_t - \frac{1}{2} \sum_{i,j = 1}^{n^C} \langle \dd \langle \dd \varphi,\mathrm{V}_j^C u_t^j \rangle,\mathrm{V}_i^C u_t^i \rangle (X_t) d\langle Z^{C;i},Z^{C;j}\rangle_t\,.
\end{split}
\end{equation}

Using the fact that for the \textit{Lie derivative} of a function $\varphi$ along a \textit{vector field} $\mathrm{V}$ it holds, see Remark \ref{REM:LieDer}, 
\[
\La_{\mathrm{V}} \varphi = \ii_{\mathrm{V}}\dd \varphi = \langle \dd \varphi,\mathrm{V}\rangle \,,
\]
and considering $\varphi(X_t)=X_t$ in equation \eqref{EQN:FinalITOSPHS} we get finally
\begin{equation}\label{EQN:FinalITOSPHS}
\begin{split}
dX_t &= \mathrm{V}^S(X_t) dZ_t +  \La_{\mathrm{V}^S} \mathrm{V}^S(X_t) d\langle  Z,Z \rangle_t+\\
&-\sum_{i= 1}^{n^{N}}\mathrm{V}^{N}_i (X_t) dZ^{N}_t- \frac{1}{2} \sum_{i,j = 1}^{n^{N}} \La_{\mathrm{V}^{N}_j} \mathrm{V}^{N}_i (X_t) d\langle Z^{N;i},Z^{N;j}\rangle_t+\\
&-\sum_{i= 1}^{n^C} \mathrm{V}_i^C  (X_t)u_t^i dZ^{C;i}_t - \frac{1}{2} \sum_{i,j = 1}^{n^C} \La_{\mathrm{V}^C_j u_t^j} \mathrm{V}^C_i (X_t) u_t^i d\langle Z^{C;i},Z^{C;j}\rangle_t\,,
\end{split}
\end{equation}
and the claim is proved.
\end{proof}

\subsection{Passivity and power--preserving property}

%As mentioned briefly above Stratonovich and It\^{o} integrals enjoy different properties; in fact on one side \textit{Stratonovich integral}, since it satisfies standard rule of differential calculus, is suited for treating geometric aspects, on the other side the martingale property of \textit{It\^{o} integral} is fundamental for probabilistic analysis. In particular this last fact will play a fundamental role when one is investigating energy conservation and stability of the underlying stochastic system.

We are now ready to address the energy conservation property, and more important in port-Hamiltonian systems, the passivity property. In particular, when one generalizes a deterministic input--output system to the stochastic case, the standard notion of passivity has several possible generalizations, leading to different possible definitions. 

\begin{Definition}[Strong/weak passivity]\label{DEF:SWPass}
Let $H \in C^\infty (\XX)$ be the total energy of the system. The SPHS \eqref{EQN:ESPHSDiss} $X$ is strongly, resp. weakly, passive if
\begin{equation}\label{EQN:SPass}
H(X_t) \leq H(X_0) + \int_0^t u^T(s)y(s)\delta Z^C_s\,,
\end{equation}
resp.
\begin{equation}\label{EQN:WPass}
\mathbb{E} H(X_t) \leq \mathbb{E}H(X_0) + \mathbb{E}\int_0^t u^T(s)y(s)\delta Z^C_s\,,
\end{equation}
for all $t \geq 0$.
\end{Definition}

\begin{Remark}
In equation \eqref{EQN:WPass}, we \fc{introduced} the passivity property in the sense of Stratonovich integration. The choice is motivated by the intuition behind the definition of SPHS.\fc{In fact, passivity means} that the total energy variation is equal or less to the total power supplied to the system, integrated along system trajectories. Since the system is formulated in terms of Stratonovich integral, we have to therefore consider the total power supplied perturbed by the corresponding control semimartingale $Z^C$, in the sense of Stratonovich.

\fc{Nonetheless, we woudl like to underline that, particularly when one considers weak passivity}, computing the expectation of a Stratonovich integral might be difficult. Therefore, as standard when dealing with energy conservation in stochastic Hamiltonian dynamics, the easiest way is to reformulate the Stratonovich integral in terms of the It\^{o} integral so that one can exploit the good probabilistic properties of the It\^{o} integral. 

Assuming for instance the SPHS to be given as \eqref{EQN:CondDS}, so that it can be converted into the equivalent formulation in terms of It\^{o} integral using Proposition \ref{PRO:SPHSTOIto}, we thus have
\[
\begin{split}
\mathbb{E} H(X_t) &\leq \mathbb{E}H(X_0) + \mathbb{E}\int_0^t u(X_s)y(s) dZ^C_s+\\
&+ \frac{1}{2} \mathbb{E}\int_0^t \partial_x \left (g(X_s)u(X_s)\right ) g(X_s)u(X_s) d \langle Z^C,Z^C \rangle_s \,,
\end{split}
\]
where we have denoted by $u$ the control in feedback form as a function of the state $X$. The weakly passivity requires that the process
\[
\left (H(X_{t}) -\mathbb{E}\int_0^tu(s) \partial_x \left (g(X_s)u(X_s)\right ) g(X_s)u(X_s) d \langle Z^C,Z^C \rangle_s- \int_0^t u^T(X_s)y(s)dZ^C_s\right )_{t \in \RR^+}\,,
\]
is a super--martingale.

Clearly, in the trivial case of $Z^C_t(\omega):=t$, passivity reduces to the classical requirement
\[
\mathbb{E} H(X_t) \leq \mathbb{E} H(X_0) + \mathbb{E}\int_0^t u^T(s)y(s)ds\,.
\]
\demo \end{Remark}

As briefly mentioned, strong passivity is a too strong assumption in many concrete situations. In fact, the presence of an external random noise implies that the system does not in general conserve energy. \cite{Ort} shows that, for Hamiltonian dynamics on a Poisson manifold, if the random perturbations are involution w.r.t. the energy of the system $H$, that is $\PS{H}{H_{N}}=0$, then the Hamiltonian system preserves the energy. In the present case, since we are considering Hamiltonian dynamics driven by the Leibniz bracket, even requiring that the stochastic Hamiltonian is an involution w.r.t the energy of the system $H$, will not ensure neither energy conservation nor passivity of the system.

Differently to the deterministic case, where considering the dissipation matrix $R$ to be positive definite allows to conclude that the PHS is passive, in the stochastic setting this is not the case since the noise driving the system may lead to an increase of the internal energy. To see that, it is enough to consider the energy conservation relation for a SPHS as given in Proposition \ref{PRO:LocVSPHS} with $\xi = 0$
\begin{equation}\label{EQN:EnBalanceNP}
H(X_t)-H(X_0) = - \int_0^t \partial_x^T H(X_s) R(X_s) \partial_x H(X_s) \delta Z_s + \int_0^t u^T y \delta Z_s^C\,.
\end{equation}
Therefore even requiring $R$ to be strictly positive definite we could not infer from equation \eqref{EQN:EnBalanceNP} the strongly passive condition
\[
H(X_t)-H(X_0) \leq  \int_0^t u^T y \delta Z_s^C\,,
\]
because,\fc{ also by just considering the trivial case of diffusive semimartingale $Z: (\omega,t)\mapsto (t+ W_t(\omega))$, $W_t$ being a standard Brownian motion, the passivity fails.}

In order to guarantee strong passivity for the SPHS we would have to require the stronger condition
\[
\int_0^t \partial_x^T H(X_s) R(X_s) \partial_x H(X_s) \delta Z_s (\omega) \geq 0\,,
\]
along all possible realizations $\omega$.

\fc{This is the main motivation why energy dissipation is usually required to hold in mean value} instead of $\omega-$wise. In fact, a positivity condition on the structural matrix $R$ does not guarantee passivity, but the requirement
\[
\mathbb{E}\int_0^t \partial_x^T H(X_s) R(X_s) \partial_x H(X_s) \delta Z_s \geq 0\,,
\]
is satisfied by a significantly larger number of semimartingales.

We stress that, due to the generality of the present setting, a complete characterization of the passivity property of SPHS is beyond the scope of the present work; in particular passivity will be object of further development in a future study. Nonetheless, next examples show that our definition of (weak) passivity leads, with some simplifications, to existent definitions of stochastic passivity.
\begin{Example}
\begin{description}[style=unboxed,leftmargin=0cm]
\item[(i)] Let consider the particular case of $Z_t(\omega) := t$, $Z^C_t(\omega):=t$ and $Z^{N}_t(\omega):= W_t(\omega)$ with $W_t$ a standard Brownian motion. Then by \cite[Th. 32]{Pro} it follows that the input--output SPHS 
\begin{equation}\label{EQN:SPHSExample}
\begin{cases}
\delta X_t &= \left [ \left [J(X_t) - R(X_t)\right ]\partial_x H(X_t) + g(X_t) u_t\right ) \delta t + \xi(X_t) \delta W_t \,,\\
%y^{N}_t &= \xi(X_t)\partial_x H(X_t)\,,\\
y_t &= g^T(X_t) \partial_x H(X_t)\,.\\
\end{cases}
\end{equation}
is a Markov process. Using Proposition \ref{PRO:SPHSTOIto} we can derive the corresponding It\^{o} formulation to be
\begin{equation}\label{EQN:SPHSExampleIto}
\begin{cases}
d X_t &= \left [ \left [J(X_t) - R(X_t)\right ]\partial_x H(X_t) + g(X_t) u_t 
+ \frac{1}{2} \left (\partial_x \xi(X_t)\right )\xi(X_t) \right ] d t + \xi(X_t) d W_t \,,\\
y_t &= g^T(X_t) \partial_x H(X_t)\,.\\
\end{cases}
\end{equation}

Using the martingale property of $W$ together with It\^{o} formula, we obtain the relation
\[
\mathbb{E}X_t - X_0 = \int_0^t \mathbb{E} \mathcal{L}H(X_t)ds\,,
\]
being $\mathcal{L}$ the infinitesimal generator associated to $X$ defined as
\begin{equation}\label{EQN:IGSPHSExample}
\begin{split}
\mathcal{L}H(x) &:= \partial_x^T H(x) \left (\left [J(x) - R(x)\right ]\partial_x H(x) + g(x) u_t 
+ \frac{1}{2} \left (\partial_x \xi(x)\right )\xi(x) \right )+\frac{1}{2} \xi^2(x) \partial^2_x H(x)\,,
\end{split}
\end{equation}
see, e.g., \cite{Oks,Pro}.

By requiring 
\begin{equation}\label{EQN:WeakPassCond}
\begin{split}
&\partial_x^T H(x) \left (R(x)\partial_x H(x) - \frac{1}{2} \left (\partial_x \xi(x)\right )\xi(x) \right ) - \frac{1}{2} \xi^2(x) \partial^2_x H(x)\geq 0\,,
\end{split}
\end{equation}
we would obtain \fc{the system to be weakly passive}. \fc{Let us underline that} condition \eqref{EQN:WeakPassCond} goes in the direction highlighted in equation \eqref{EQN:DoubleDiss} where a dissipation condition is imposed jointly on the resistive and stochastic ports.

\item[(ii)] Let consider stochastic perturbations for both resistive and storage port, with $Z_t(\omega):= t + B_t$, where $B$ is a standard Brownian motion independent of $W$. Therefore equation \eqref{EQN:SPHSExampleIto} becomes
\begin{equation}\label{EQN:SPHSExampleIto2}
\begin{cases}
d X_t &= \left [ \left [J(X_t) - R(X_t)\right ]\partial_x H(X_t) + g(X_t) u + \frac{1}{2} \left (\partial_x \xi(X_t)\right )\xi(X_t) \right ] d t+ \\
&+\frac{1}{2} \left (\partial_x \left [J(X_t) - R(X_t)\right ]\partial_x H(X_t)\right )\left [J(X_t) - R(X_t)\right ]\partial_x H(X_t) d t +\\
&\qquad + \left [J(X_t) - R(X_t)\right ]\partial_x H(X_t) dB_t + \xi(X_t) d W_t \,,\\
%y^{N}_t &= \xi(X_t)\partial_x H(X_t)\,,\\
y_t &= g^T(X_t) \partial_x H(X_t)\,,\\
\end{cases}
\end{equation}
and denoting for short by $\mu$ the drift in equation \eqref{EQN:SPHSExampleIto2}, we obtain the infinitesimal generator of the form
\[
\begin{split}
\mathcal{L}H(x) &:= \partial_x^T H(x) \mu(x) + \frac{1}{2} \left (\left [J(x) - R(x)\right ]\partial_x H(x) \right )^2 \partial^2_x H(x) +\frac{1}{2} \xi^2(x) \partial^2_x H(x) \,.
\end{split}
\]
The condition 
\[
\begin{split}
\mathcal{L}H(x) \leq 0\,,
\end{split}
\]
guarantees that $X$ is weakly passive.
\end{description}
\demo
\end{Example}

\subsection{Interconnection of Stochastic port-Hamiltonian systems}\label{SEC:IntSPHS}

The present Section is devoted to prove that the composition of $N$ Dirac structures is again a Dirac structure. We will start showing that the composition of two Dirac structures is again a Dirac structure; clearly this immediately generalize by induction to the fact the composition of $N$ Dirac structures is again a Dirac structure.

Let $\mathcal{D}_A \subset \TT \XX_A \times \TT^* \XX_A \times \mathcal{F} \times \mathcal{E}$ and $\mathcal{D}_B \subset \TT \XX_B \times \TT^* \XX_B \times \mathcal{F} \times \mathcal{E}$ be two Dirac structures perturbed by two semimartingales $Z_A$ and $Z_B$, being $\XX_A$ and $\XX_B$ two general manifolds. The particular form for the Dirac structures implies that $\mathcal{D}_A$ and $\mathcal{D}_B$ shares a common port $\mathcal{F} \times \mathcal{E}$, through which they are connected.

Equating flows and efforts through the shared port $\mathcal{F} \times \mathcal{E}$, i.e.
\[
\delta f_A = - \delta f_B \,,\quad e_A = e_B\,,
\]
where $\delta f_A$ and $e_A$ are the flow and effort connected to the port $\mathcal{F} \times \mathcal{E}$ in $\mathcal{D}_A$ and similarly holds for $\mathcal{D}_B$, we have that the composition of the two Dirac structures, i.e.,
\begin{equation}\label{EQN:CompDS}
\begin{split}
\mathcal{D}_A \circ \mathcal{D}_B :=  \{&\left (\delta X^A_t,\theta^A,\delta X_t^B,\theta^B\right ) \in \TT \XX_A \times \TT^* \XX_A \times \TT \XX_B \times \TT^* \XX_B \\
&\left (\delta X^A_t,\theta^A,- \delta f_B ,e_B\right ) \in \mathcal{D}_A \, \mbox{and} \, \left (\delta X^B_t,\theta^B,\delta f_B ,e_B\right ) \in \mathcal{D}_B \}\,,
\end{split}
\end{equation}
is again a Dirac structure, see Figure \ref{FIG:InterPHS2} for a graphical representation.

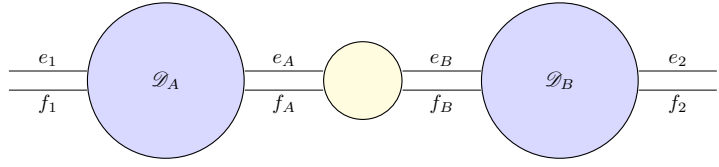
\begin{figure}
\centering
\resizebox {0.7\textwidth} {!} {
\begin{tikzpicture}[
  font=\sffamily,
  every matrix/.style={ampersand replacement=\&,column sep=2cm,row sep=2cm},
  source/.style={draw,thick,circle,fill=yellow!15,inner sep=.3cm ,minimum size=20mm},
  process/.style={draw,thick,circle,fill=blue!15,minimum  size=40mm},
  sink/.style={fill=green!0},
  datastore/.style={draw,very thick,shape=datastore,inner sep=.3cm},
  dots/.style={gray,scale=2},
  to/.style={->,>=stealth',shorten >=1pt,semithick,font=\sffamily\footnotesize},
  every node/.style={align=center}]

  % Position the nodes using a matrix layout
  \matrix{\node[sink] (sx) {}; \& \node[process] (uno) {\scalebox{1.5}{$\mathcal{D}_A$}}; \& \node[source] (d) {}; \&  \node[process] (due) {\scalebox{1.5}{$\mathcal{D}_B$}}; \& \node[sink] (dx) {} ; \\ };

  % Draw the arrows between the nodes and label them.
  \draw[transform canvas={yshift=-7pt}] (uno) -- node[midway,below] {\scalebox{1.5}{$f_1$}} (sx);
  \draw[transform canvas={yshift=7pt}] (uno) -- node[midway,above] {\scalebox{1.5}{$e_1$}} (sx);
  \draw[transform canvas={yshift=-7pt}] (uno) -- node[midway,below] {\scalebox{1.5}{$f_A$}} (d);
  \draw[transform canvas={yshift=7pt}] (uno) -- node[midway,above] {\scalebox{1.5}{$e_A$}} (d);
  \draw[transform canvas={yshift=-7pt}] (due) -- node[midway,below] {\scalebox{1.5}{$f_B$}} (d);
  \draw[transform canvas={yshift=7pt}] (due) -- node[midway,above] {\scalebox{1.5}{$e_B$}} (d);
  \draw[transform canvas={yshift=-7pt}] (due) -- node[midway,below] {\scalebox{1.5}{$f_2$}} (dx);
  \draw[transform canvas={yshift=7pt}] (due) -- node[midway,above] {\scalebox{1.5}{$e_2$}} (dx);
\end{tikzpicture}
}
\caption{Interconnection of implicit port-Hamiltonian system.}\label{FIG:InterPHS2}
\end{figure}

We have the following result.

\begin{Proposition}\label{PRO:Comp}
Let $\mathcal{D}_A \subset \TT \XX_A \times \TT^* \XX_A \times \mathcal{F} \times \mathcal{E}$ and $\mathcal{D}_B \subset \TT \XX_B \times \TT^* \XX_B \times \mathcal{F} \times \mathcal{E}$ two Dirac structures, then $\mathcal{D}_A \circ \mathcal{D}_B \subset \TT \XX_A \times \TT^* \XX_A \times \TT \XX_B \times \TT^* \XX_B$ as defined in equation \eqref{EQN:CompDS} is a Dirac structure.
\end{Proposition}
\begin{proof}
In what follows we will denote for short $\mathcal{D}:= \mathcal{D}_A \circ \mathcal{D}_B$.

Let us first prove that $\mathcal{D} \subset \mathcal{D}^\perp$. 

Let $\left (\delta X^A_t,\theta^A,\delta X_t^B,\theta^B\right ) \in \mathcal{D}$, then for any $\left (\delta \bar{X}^A_t,\bar{\theta}^A,\delta \bar{X}_t^B,\bar{\theta}^B\right ) \in \mathcal{D}$ we have that
\begin{equation}\label{EQN:S1}
\begin{split}
\langle \langle \left (\delta X^A_t,\theta^A,\delta X_t^B,\theta^B\right ),\left (\delta \bar{X}^A_t,\bar{\theta}^A,\delta \bar{X}_t^B,\bar{\theta}^B\right )\rangle\rangle &= \langle \delta X^A_t,\bar{\theta}^A\rangle + \langle \delta \bar{X}^A_t,\theta^A\rangle +\\
&+\langle \delta X^B_t,\bar{\theta}^B\rangle + \langle \delta \bar{X}^B_t,\theta^B\rangle\,.
\end{split}
\end{equation}

Since we have that $\left (\delta X^A_t,\theta^A,\delta X_t^B,\theta^B\right )$, $\left (\delta \bar{X}^A_t,\bar{\theta}^A,\delta \bar{X}_t^B,\bar{\theta}^B\right ) \in \mathcal{D}$, there exist $(\delta f,e)$ and $(\delta \bar{f},e)$ such that
\begin{equation}\label{EQN:S1b}
\begin{split}
\left (\delta X^A_t,\theta^A,-\delta f,e\right ) \in &\mathcal{D}_A\,,\quad \left (\delta X_t^B,\theta^B,\delta f,e)\right )\in \mathcal{D}_B\,,\\
\left (\delta \bar{X}^A_t,\bar{\theta}^A,-\delta \bar{f},\bar{e}\right ) \in &\mathcal{D}_A\,,\quad \left (\delta \bar{X}_t^B,\bar{\theta}^B,\delta \bar{f},\bar{e})\right )\in \mathcal{D}_B\,.\\
\end{split}
\end{equation}

Therefore equation \eqref{EQN:S1} can be rewritten as
\begin{equation}\label{EQN:S1c}
\begin{split}
&\langle \delta X^A_t,\bar{\theta}^A\rangle + \langle \delta \bar{X}^A_t,\theta^A\rangle +\langle \delta X^B_t,\bar{\theta}^B\rangle + \langle \delta \bar{X}^B_t,\theta^B\rangle = \\
=&\langle \delta X^A_t,\bar{\theta}^A\rangle + \langle \delta \bar{X}^A_t,\theta^A\rangle - \langle \delta f,\bar{e}\rangle-\langle \delta \bar{f},e \rangle +\\
+&\langle \delta X^B_t,\bar{\theta}^B\rangle + \langle \delta \bar{X}^B_t,\theta^B\rangle+\langle \delta f,\bar{e}\rangle +\langle \delta \bar{f},e \rangle =0\,,
\end{split}
\end{equation}
where the last equality follows from equation \eqref{EQN:S1b} together with the fact that $\mathcal{D}_A$ and $\mathcal{D}_B$ are Dirac structures. Therefore, we have that $\left (\delta X^A_t,\theta^A,\delta X_t^B,\theta^B\right ) \in \mathcal{D}^\perp$ and $\mathcal{D} \subset \mathcal{D}^\perp$ is proved.

Let us now prove conversely that $\mathcal{D}^\perp \subset \mathcal{D}$.

Let $\left (\delta X^A_t,\theta^A,\delta X_t^B,\theta^B\right ) \in \mathcal{D}^\perp$, then
\begin{equation}\label{EQN:S2a}
\begin{split}
0&=\langle \delta X^A_t,\bar{\theta}^A\rangle + \langle \delta \bar{X}^A_t,\theta^A\rangle +\langle \delta X^B_t,\bar{\theta}^B\rangle + \langle \delta \bar{X}^B_t,\theta^B\rangle\,,
\end{split}
\end{equation}
for all $\left (\delta \bar{X}^A_t,\bar{\theta}^A,\delta \bar{X}_t^B,\bar{\theta}^B\right ) \in \mathcal{D}$, that is there exist $\delta \bar{f}$ and $e$ such that
\[
\begin{split}
\left (\delta \bar{X}^A_t,\bar{\theta}^A,-\delta \bar{f},\bar{e}\right ) \in \mathcal{D}_A\,,\quad \left (\delta \bar{X}^B_t,\bar{\theta}^B,\delta \bar{f},\bar{e}\right ) \in \mathcal{D}_B\,.
\end{split}
\]

By choosing $\delta \bar{X}^B = \bar{\theta}^B = 0$ equation \eqref{EQN:S2a} becomes
\[
\langle \delta X^A_t,\bar{\theta}^A\rangle + \langle \delta \bar{X}^A_t,\theta^A\rangle = 0\,.
\]
If $\left (\delta \bar{X}^A_t,\bar{\theta}^A,-\delta \bar{f},\bar{e}\right ) \in \mathcal{D}_A$ and $\left ((\delta \bar{X}^A_t)',(\bar{\theta}^A)',-\delta \bar{f},\bar{e}\right ) \in \mathcal{D}_A$, then we have
\[
\left (\delta \bar{X}^A_t-(\delta \bar{X}^A_t)',\bar{\theta}^A-(\bar{\theta}^A)',0,0\right ) \in \mathcal{D}_A\,.
\]
Defining the linear operator $T^A$ as 
\[
T^A \left (\delta \bar{X}^A_t,\bar{\theta}^A,-\delta \bar{f},\bar{e}\right ) := \langle \delta X^A_t,\bar{\theta}^A\rangle + \langle \delta \bar{X}^A_t,\theta^A\rangle\,,
\]
we have that
\[
T^A \left (\delta \bar{X}^A_t-(\delta \bar{X}^A_t)',\bar{\theta}^A-(\bar{\theta}^A)',0,0\right )=0\,,
\]
so that the linearity of $T$ in turn implies
\[
T^A \left (\delta \bar{X}^A_t,\bar{\theta}^A,-\delta \bar{f},\bar{e}\right )= T^A \left ((\delta \bar{X}^A_t)',(\bar{\theta}^A)',-\delta \bar{f},\bar{e}\right )\,.
\]
\fc{Consequently, by the linearity of $T^A$, we infer} that there exists $\delta f^A$ and $e^A$ such that
\[
T^A \left (\delta \bar{X}^A_t,\bar{\theta}^A,-\delta \bar{f},\bar{e}\right ) = \langle \delta f^A,\bar{e}\rangle + \langle \delta \bar{f},e^A\rangle \,,
\]
or equivalently using the definition of $T^A$ we have that there exist $\delta f^A$ and $e^A$ such that
\begin{equation}\label{EQN:S2b}
\langle \delta X^A_t,\bar{\theta}^A\rangle + \langle \delta \bar{X}^A_t,\theta^A\rangle + \langle \delta f^A,\bar{e}\rangle - \langle \delta \bar{f},e^A\rangle = 0\,.
\end{equation}

Repeating the same reasoning, choosing $\delta \bar{X}^A = \bar{\theta}^A = 0$ we obtain
\begin{equation}\label{EQN:S2c}
\langle \delta X^B_t,\bar{\theta}^B\rangle + \langle \delta \bar{X}^B_t,\theta^B\rangle + \langle \delta f^B,\bar{e}\rangle + \langle \delta \bar{f},e^B\rangle = 0\,.
\end{equation}

Substituting now equations \eqref{EQN:S2b}--\eqref{EQN:S2c} into equation \eqref{EQN:S2a} we get
\begin{equation}\label{EQN:S2d}
0=\langle \delta \bar{f},e^A\rangle - \langle \delta f^A,\bar{e}\rangle -\langle \delta \bar{f},e^B\rangle - \langle \delta f^B,\bar{e}\rangle = \langle \delta \bar{f},e^A-e^B \rangle - \langle \delta f^A + \delta f^B,\bar{e}\rangle\,,
\end{equation}
so that we can conclude that $\delta f^A = -\delta f^B$ and $e^A = e^B$, and therefore we have shown that $\mathcal{D}^\perp \subset \mathcal{D}$ and the proof is complete.
\end{proof}

This proposition can be generalized to consider $N$ implicit SPHS with state space $\XX_i$, Hamiltonian $H_i$ and flows effort space $\mathcal{F}_i \times \mathcal{E}_i$, by defining the interconnection Dirac structure as
\[
\mathcal{D}_I \subset \bigtimes_{i=1}^N \left ( \mathcal{F}_i \times \mathcal{E}_i \times \mathcal{F}\times \mathcal{E}\right )\,.
\]

We thus have that $\mathcal{D}:= \bigtimes_{i=1}^N \mathcal{D}_i$ is a Dirac structure on $\XX := \bigtimes_{i=1}^N \XX_i$, so that by Proposition \ref{PRO:Comp} we have that $\mathcal{D} \circ \mathcal{D}_I$ is a Dirac structure on $\XX$; Figure \ref{FIG:InterPHS} shows a representation of interconnected port-Hamiltonian systems.

\begin{Proposition}
The interconnection of $N$ SPHS with state space $\XX_i$, Hamiltonian $H_i$ and flows effort space $\mathcal{F}_i \times \mathcal{E}_i$ connected through an interconnection Dirac structure $\mathcal{D}_I$ and perturbing semimartingale $\mathbf{Z}^i$, defines a SPHS with Dirac structure $\mathcal{D} \circ \mathcal{D}_I$ and Hamiltonian $H:= \sum_{i=1}^N H_i$.
\end{Proposition}

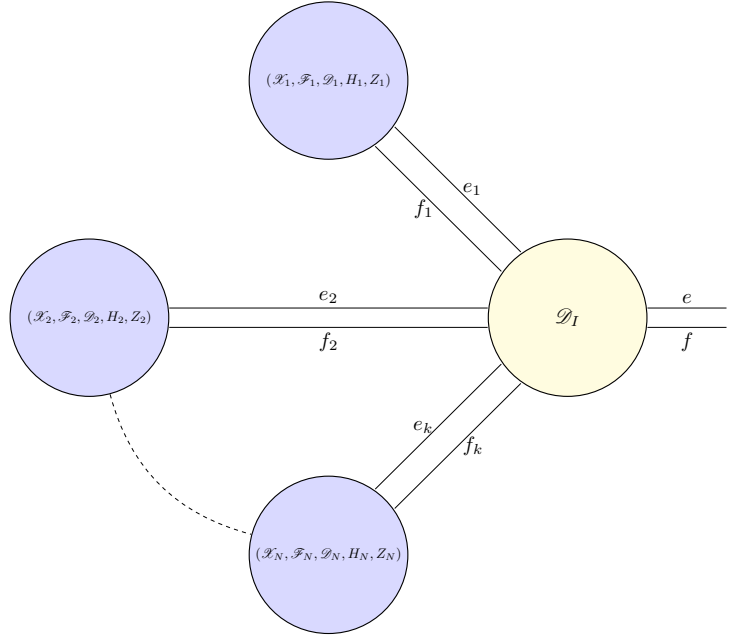
\begin{figure}
\centering
\resizebox {0.7\textwidth} {!} {
\begin{tikzpicture}[
  font=\sffamily,
  every matrix/.style={ampersand replacement=\&,column sep=2cm,row sep=2cm},
  source/.style={draw,thick,circle,fill=yellow!15,inner sep=.3cm ,minimum size=40mm},
  process/.style={draw,thick,circle,fill=blue!15,minimum  size=40mm},
  sink/.style={fill=green!0},
  datastore/.style={draw,very thick,shape=datastore,inner sep=.3cm},
  dots/.style={gray,scale=2},
  to/.style={->,>=stealth',shorten >=1pt,semithick,font=\sffamily\footnotesize},
  every node/.style={align=center}]

  % Position the nodes using a matrix layout
  \matrix{ \& \node[process] (uno) {$(\XX_1,\mathcal{F}_1,\mathcal{D}_1,H_1,Z_1)$}; \& \& \\

 \node[process] (due) {$(\XX_2,\mathcal{F}_2,\mathcal{D}_2,H_2,Z_2)$}; \& \& \node[source] (d) {\scalebox{1.5}{$\mathcal{D}_I$}}; \&\node[sink] (ext) {}; \\

 \& \node[process] (k) {$(\XX_N,\mathcal{F}_N,\mathcal{D}_N,H_N,Z_N)$}; \& \&\\ };

  % Draw the arrows between the nodes and label them.
  \draw[transform canvas={xshift=-7pt,yshift=-7pt}] (uno) -- node[midway,left] {\scalebox{1.5}{$f_1$}} (d);
  \draw[transform canvas={xshift=7pt,yshift=7pt}] (uno) -- node[midway,right] {\scalebox{1.5}{$e_1$}} (d);
  \draw[transform canvas={yshift=-7pt}] (due) -- node[midway,below] {\scalebox{1.5}{$f_2$}} (d);
  \draw[transform canvas={yshift=7pt}] (due) -- node[midway,above] {\scalebox{1.5}{$e_2$}} (d);
  \draw[transform canvas={xshift=7pt,yshift=-7pt}] (k) -- node[midway,right] {\scalebox{1.5}{$f_k$}} (d);
  \draw[transform canvas={xshift=-7pt,yshift=7pt}] (k) -- node[midway,left] {\scalebox{1.5}{$e_k$}} (d);
  \draw[transform canvas={yshift=-7pt}] (d) -- node[midway,below] {\scalebox{1.5}{$f$}} (ext);
  \draw[transform canvas={yshift=7pt}] (d) -- node[midway,above] {\scalebox{1.5}{$e$}} (ext);
  \draw[dashed] (due) to[bend right] (k);
  
\end{tikzpicture}
}
\caption{Interconnection of implicit port-Hamiltonian system.}\label{FIG:InterPHS}
\end{figure}

\begin{Example}\label{EX:InterSPHS}
The interconnection of two port Hamiltonian systems of the form given in equation \eqref{EQN:CondDS}, 
\[
\begin{split}
&\begin{cases}
\delta X_t &= \left (J(X_t) - R(X_t)\right )\partial_x H(X_t) \delta Z_t - \sum_{i=1}^m \xi_i(X_t) \delta Z^{N;i}_t - g(X_t) u_t \delta Z^C_t\,,\\
%y^i_t &= g^T_i(X_t)\partial H(X_t)\,,\\
y_t &= g^T(X_t) \partial_x H(X_t)\,,\\
\end{cases}\,,\\
&\begin{cases}
\delta \bar{X}_t &= \left (\bar{J}(\bar{X}_t) - \bar{R}(\bar{X}_t)\right )\partial_{\bar{x}} \bar{H}(\bar{X}_t) \delta \bar{Z}_t - \sum_{i=1}^{\bar{m}} \bar{\xi}_i(\bar{X}_t) \delta \bar{Z}^{N;i}_t - \bar{g}(\bar{X}_t) \bar{u} \delta \bar{Z}^C_t\,,\\
%\bar{y}^i_t &= \bar{g}^T_i(\bar{X}_t)\partial \bar{H}(\bar{X}_t)\,,\\
\bar{y}_t &= \bar{g}^T(X_t) \partial_{\bar{x}} \bar{H}(\bar{X}_t)\,,\\
\end{cases}
\end{split}
\]
through the power preserving connection
\[
u = - \bar{u}\,,\quad y=\bar{y}\,,
\]
leads to a stochastic PHS of the form
\[
\begin{cases}
\delta X_t &= \left (J(X_t) - R(X_t)\right )\partial_x H(X_t) \delta Z_t - \sum_{i=1}^m \xi_i(X_t) \delta Z^{N;i}_t - g(X_t) \lambda \delta Z^C_t\,,\\
\delta \bar{X}_t &= \left (\bar{J}(\bar{X}_t) - \bar{R}(\bar{X}_t)\right )\partial_{\bar{x}} \bar{H}(\bar{X}_t) \delta \bar{Z}_t - \sum_{i=1}^{\bar{m}} \bar{\xi}_i(\bar{X}_t) \delta \bar{Z}^{N;i}_t + \bar{g}(\bar{X}_t) \lambda \delta \bar{Z}^C_t\,,\\
g^T(X_t)\partial_x H(X_t) &= \bar{g}^T(X_t) \partial_{\bar{x}} \bar{H}(\bar{X}_t)\,.
\end{cases}
\]
\demo \end{Example}

Example \ref{EX:InterSPHS} can be generalized to consider $N$ explicit semimartingale port-Hamiltonian systems in local coordinates of the form \eqref{EQN:CondDS}, $i = 1, \dots,N$, on a general $n_i-$dimensional manifold $\XX_i$. 

In general, we could consider a power-preserving interconnection of the SPHS, that is a subspace
\[
I(X^1_t,\dots,X^N_t) \subset \mathcal{F}^1 \times \dots \times \mathcal{F}^N \times \mathcal{E}^1\times \dots \mathcal{E}^N\,,
\]
such that power is preserved, namely
\begin{equation}\label{EQN:PowC}
(\delta f^1_t,\dots,\delta f^N_t,e_t^1,\dots,e_t^N)\in I \Rightarrow \sum_{i=1}^N \int_0^t  \langle e^i_s,\delta f^i_s\rangle =0\,.
\end{equation}

Notice that the interconnection $I$, as given above, defines a Dirac structure.

\color{black}
\subsection{Examples}
\subsubsection{The mass-spring system}\label{EX:massspring}
Consider the \textit{mass--spring system}
\begin{equation}\label{EQN:mss}
m \ddot{x} = -k x + F\,,
\end{equation}
where $x$ is the position of the system, $m$ its mass, $F$ the applied force and $k$ the stiffness of the spring. Defining $p = m \dot{x}$ as the momentum and $q=x$, it is easily seen that $X= (p,q)$ defines a PHS with respect to the energy
\[
H(p,q)= \frac{1}{2}k q^2 + \frac{1}{2}\frac{p^2}{m}\,,
\]
of the form
\[
\begin{cases}
\dot{X} &= J \partial_x H(X) + g F\,,\\
y&=g^T\partial H(X)\,,
\end{cases}
\]
with
\[
J = 
\begin{pmatrix}
0 & 1 \\ -1 & 0
\end{pmatrix}\,,\quad
g = \begin{pmatrix}
0 \\ 1
\end{pmatrix} \,,\quad
\partial_x H(X) = \begin{pmatrix}
k q \\\frac{p}{m}
\end{pmatrix}\,.
\]

Let $Z_t(\omega) = t + W_t$, being $W_t$ a standard Brownian motion, we can generalize equation \eqref{EQN:mss} to consider a stochastic term
\begin{equation}\label{EQN:msss}
\binom{\delta q_t}{\delta p_t} = \binom{\frac{p_t}{m}}{- k q_t + F}\delta t + \binom{\frac{p_t}{m}}{- k q_t} \delta W_t\,;
\end{equation}
or in It\^{o} form 
\begin{equation}\label{EQN:msss}
\binom{d q_t}{d p_t} = \binom{\frac{p_t}{m}-\frac{kq_t}{2m}}{- k q_t - \frac{k}{2m}p_t + F}d t + \binom{\frac{p_t}{m}}{- k q_t} d W_t\,.
\end{equation}
Denoting by 
\[
\bar{q}_t := \mathbb{E}q_t\,,\quad \bar{p}_t := \mathbb{E}p_t\,,
\]
we obtain, using the fact that the integral w.r.t. $W_t$ is a martingale,
\[
\begin{split}
\dot{\bar{q}}_t &= \frac{\bar{p}_t}{m}-\frac{k \bar{q}_t}{2m}\,,\\
\dot{\bar{p}}_t &= - k \bar{q}_t - \frac{k}{2m} \bar{p}_t + F\,.
\end{split}
\]
Since $p = m \dot{q}$, we have
\[
m \ddot{\bar{q}}_t = m \dot{\bar{p}}_t = - \frac{k}{2} \dot{\bar{q}}_t - k \bar{q}_t + F \,,
\]
which is the equation for a damped harmonic oscillator.

%More generally we can consider a mass--spring system with external noise and $Z^N_t(\omega) = B_t(\omega)$, being $B_t$ a standard Brownian motion independent of $W_t$,
%\begin{equation}\label{EQN:msssEN}
%\binom{\delta q_t}{\delta p_t} = \binom{\frac{p_t}{m}}{- k q_t+F}\delta t + \binom{\frac{p_t}{m}}{- k q_t} \delta W_t + \binom{\xi^q}{\xi^p} \delta B_t \,,
%\end{equation}
%and considerations above follow similarly.

\subsubsection{The $n$-DOF robotic arm}\label{EX:nDOFNoise}
Consider a \textit{n-DOF} fully actuated mechanical system with generalized coordinate $q$, see \cite{Sec} for the deterministic treatment; let $p = M(q)\dot{q}$ be the generalized momenta, and $H(p,q)$ be the Hamiltonian
\[
H(p,q) = \frac{1}{2}p^TM^{-1}(q)p+V(q)\,,
\]
with the structure matrices
\[
J= 
\begin{pmatrix}
0 & I_n\\
-I_n & 0
\end{pmatrix}\,,\quad 
R=
\begin{pmatrix}
0 & 0 \\
0 & D(p,q)
\end{pmatrix}\,,\quad 
g = 
\begin{pmatrix}
0 \\
B(q)
\end{pmatrix}\,,\quad
S = J-R\,.
\]

The stochastic $n$-DOF robot with model noise is
\begin{equation}\label{EQN:StochnDOF}
\begin{cases}
\delta X_t &= S(X_t) \partial_x H(X_t) \delta t +  S(X_t) \partial_x H(X_t) \delta W_t + g(X_t) u_t \delta t + \xi(X_t) \delta B_t\,,\\
y_t &= G^T(X_t) \partial_x H(X_t)\,,
\end{cases}
\end{equation}
with $X_t = (p_t,q_t)$ and where we considered the semimartingale $Z_t(\omega):= t + \sigma W_t(\omega)$, being $\sigma>0$ and $W_t$ a standard Brownian motion, $Z^{N}_t(\omega):= B_t$, with $B_t$ a standard Brownian motion independent of $W_t$, and $Z^C_t(\omega) = t$.

Equivalently in It\^{o} form we get
\begin{equation}\label{EQN:StochnDOFIto}
\begin{cases}
d X_t &= \left (S(X_t) \partial_x H(X_t) + \sigma^2 \partial_x\left (S(X_t) \partial_x H(X_t) \right ) S(X_t) \partial_x H(X_t) \right ) d t +\\
&+ \partial_x (\xi(X_t)) \xi(X_t)dt + g(X_t) u_t dt + S(X_t) \partial_x H(X_t)  d W_t + \xi(X_t) d B_t\,,\\
y_t &= g^T(X_t) \partial_x H(X_t)\,,
\end{cases}
\end{equation}

\subsubsection{The DC motor}\label{EX:nDOFNoise}
Consider a DC motor, that is $\XX = \RR^2$ and $X=(\phi,p)$ with Hamiltonian
\[
H(p,\phi)=\frac{1}{2}\frac{p^2}{I}+\frac{1}{2}\frac{\phi^2}{L};
\]
and structure matrices
\[
J= 
\begin{pmatrix}
0 & K\\
-K & 0
\end{pmatrix}\,,\quad 
R=
\begin{pmatrix}
b & 0 \\
0 & R
\end{pmatrix}\,,\quad 
g = 
\begin{pmatrix}
0 \\
1
\end{pmatrix}\,,\quad
S = J-R\,,
\]
see \cite{Sec} for the deterministic treatment.

The stochastic DC motor with noise is
\begin{equation}\label{EQN:StochnDOF}
\begin{cases}
\delta X_t &= S(X_t) \partial_x H(X_t) \delta t +  S(X_t) \partial_x H(X_t) \delta W_t + g(X_t) u_t \delta t + \xi(X_t) \delta B_t\,,\\
y_t &= g^T(X_t) \partial_x H(X_t)\,,
\end{cases}
\end{equation}
with $X_t=(p_t,\phi_t)$ and where the semimartingale $Z_t(\omega):= t + \sigma W_t(\omega)$, being $\sigma>0$ and $W_t$ a standard Brownian motion, $Z^N_t(\omega):= B_t$, with $B_t$ a standard Brownian motion independent of $W_t$, and $Z^C_t(\omega) = t$.

Equation \eqref{EQN:StochnDOF} can be rewritten in It\^{o} form as
\begin{equation}\label{EQN:StochnDOFIto}
\begin{cases}
d X_t &= \left (S(X_t) \partial_x H(X_t) + \frac{1}{2} S(X_t) \partial_x H(X_t) \partial_x\left (S(X_t) \partial_x H(X_t) \right )\right ) d t +\\
&+ S(X_t) \partial_x H(X_t)  d W_t +  g(X_t) u_t dt\,,\\
y_t &= g^T(X_t) \partial_x H(X_t)\,,
\end{cases}
\end{equation}

\subsubsection{The Van der Pol osclillator}\label{EX:nDOFNoise}

Consider a stochastic \textit{van der Pol oscillator} of the form
\[
\begin{cases}
\delta x_1=x_2\delta t\,,\\
\delta x_2(t) = \left (\mu (1-x^2_1)x_2(t) - x_1(t)\right )\delta t + \xi(x_2) \delta W_t\,.
\end{cases}
\]
or written for short as
\begin{equation}\label{EQN:vdp}
\delta X_t = \left (J(X_t)-R(X_t)\right )\partial_x H(X_t) \delta t + \xi(X_t)\delta dW_t\,,
\end{equation}
where the energy Hamiltonian function is
\[
H(X_t) = \frac{1}{2}X^T_t I X_t\,,
\]
with $I$ the $2 \times 2-$ identity matrix, and dissipation structure
\[
R(X_t) = 
\begin{pmatrix}
0 & 0\\
0 & -\mu (1-x^2_1(t))x_2(t)
\end{pmatrix}\,.
\]

This stochastic dynamics has been treated for instance in \cite{CDPvdp} or also in the more general form of a stochastic Fitz--Hugh Nagumo (FHN) model in \cite{CDPFHN,CDPFHN2}.

\color{black}

\section{Conclusions}\label{SEC:Conc}
This work is the first step of a more general research program intended to rigorously study stochastic port--Hamiltonian systems. In the present paper we formally introduced the definition of \textit{stochastic implicit port-Hamiltonian system}, showing how the considered setting generalizes existing notions of \textit{stochastic Hamiltonian dynamics} as well as \textit{deterministic port--Hamiltonian systems}. OOne of the main novelty of our approach consists in allowing the elements of the port--Hamiltonian to be stochastic vector fields, so that the power exchanged by the system is allowed to be a general semimartingale. In this sense, the noise does not enter the system solely as an external perturbation, the system itself being intrinsically stochastic. We further showed how the \textit{stochastic implicit port-Hamiltonian system} can be equivalently formulated in terms of either Stratonovich or It\^{o} integration. At last, an investigation on energy conservation, passivity and power--preserving interconnection of SPHS has been carried out.

\color{black}

\cleardoublepage
%\nocite{*}
\bibliographystyle{apalike}
\bibliography{bib}
\end{document}